\documentclass{amsart}
\usepackage[utf8]{inputenc}
\usepackage{amsthm}
\usepackage[non-compressed-cites, 
non-sorted-cites, alphabetic]{amsrefs}
\usepackage[space]{cite}
\usepackage{fullpage}
\usepackage[margin=1in]{geometry}
\usepackage{amsmath}
\usepackage{stackengine}
\usepackage{mathrsfs} 
\usepackage{amssymb}
 \usepackage{amsfonts}
 \usepackage{bbm} % for \mathbbm{1}
 \usepackage{physics} % for qand
 \usepackage{mathtools}
\usepackage{wasysym}
\usepackage{xcolor}
\usepackage[shortlabels]{enumitem}
\allowdisplaybreaks
\definecolor{winered}{rgb}{0.7,0,0} % Peter Hintz's formatting colors
\definecolor{lessblue}{rgb}{0,0,0.7} % Peter Hintz's formatting colors
\usepackage[pdftex,colorlinks=true,linkcolor=winered,citecolor=lessblue,breaklinks=true,bookmarksopen=true]{hyperref}
\usepackage{mathtools}
\definecolor{mypink1}{rgb}{0.858, 0.188, 0.478}

\numberwithin{equation}{section}
\theoremstyle{plain}
\newtheorem{Th}{Theorem}[section]
\newtheorem{Lemma}[Th]{Lemma}

\newtheorem{prop}[Th]{Proposition}
\newtheorem{definition}[Th]{Definition}
\theoremstyle{definition}

\newtheorem{remarkx}[Th]{Remark}
\newenvironment{remark}
  {\pushQED{\qed}\remarkx}
  {\popQED\endremarkx}
 \newcommand{\R}{\mathbb{R}}
\newcommand{\C}{\mathbb{C}}

\newcommand{\imag}{\operatorname{Im}}
\newcommand{\qwhere}{\qquad\operatorname{ where }\qquad}
\renewcommand{\real}{\operatorname{Re}}

\newcommand{\supp}{\operatorname{supp }}
\newcommand{\lecal}{\mathcal{L}\mathcal{E}}
\renewcommand{\norm}[1]{\left\lVert #1 \right\rVert}
\newcommand{\vertiii}[1]{{\left\vert\kern-0.25ex\left\vert\kern-0.25ex\left\vert #1 
    \right\vert\kern-0.25ex\right\vert\kern-0.25ex\right\vert}}
\newcommand{\inprod}[1]{\left\langle #1\right\rangle}
\title{Integrated Local Energy Decay for Damped Magnetic Wave Equations on Stationary Space-Times}
\author{Collin Kofroth}
\address{Johns Hopkins University Applied Physics Lab, 11100 Johns Hopkins Road, Laurel, MD 20723-6099}
\email{collin.kofroth@jhuapl.edu}
\date{\today}
\begin{document}
\begin{abstract}
    We establish local energy decay for damped magnetic wave equations on stationary, asymptotically flat space-times subject to the geometric control condition. More specifically, we allow for the addition of time-independent magnetic and scalar potentials, which negatively affect energy coercivity and may add in unwieldy spectral effects. By asserting the non-existence of eigenvalues in the lower half-plane and resonances on the real line, we are able to apply spectral theory from the work of Metcalfe, Sterbenz, and Tataru and combine with a generalization of prior work by the present author to extend the latter work and establish local energy decay, under one additional symmetry hypothesis. Namely, we assume that the damping term is the dominant principal term in the skew-adjoint part of the damped wave operator within the region where the metric perturbation from that of Minkowski space is permitted to be large.  We also obtain an energy dichotomy if we do not prohibit non-zero real resonances. In order to make the structure of the argument more cohesive, we contextualize the present work within the requisite existing theory.
\end{abstract}
\maketitle
\tableofcontents
\section{Introduction}
\subsection{Background and Problem Set-Up}\label{Background}
In this work, we will establish (integrated) local energy decay for a broad class of damped wave equations on stationary, asymptotically flat space-times, including those which possess magnetic and scalar potential terms. This generalizes the result of \cites{Kof22}, which lacked potential terms and assumed that the damping was compactly-supported. The presence of the potentials allow for interaction with the damping and the existence of both complex eigenvalues and real resonances embedded in the continuous spectrum, all of which stand to inhibit local energy decay. We will discuss each of these obstructions carefully. Along the way, we also establish an energy dichotomy when non-zero real resonances are permissible.

To set up the problem which we will study, let $(\R^{4},g)$ be a Lorentzian manifold with coordinates $(t,x)\in\R\times\R^3$ and metric signature $(-+++).$ We will consider \textit{damped wave operators} of the form $$P=\Box_{A,g}+iaD_t+V,\qquad \Box_{A,g}=(D_\alpha+A_\alpha)g^{\alpha\beta}(D_\beta+A_\beta)$$ where $a, V,$ and the components of $A$ are smooth and complex-valued; here, we are utilizing the notation $D_\alpha =-i\partial_\alpha$, which should be interpreted as an operator. We will call $A$ the magnetic potential, $V$ the scalar potential, and $a$ the damping. If $a\equiv 0$, then we will simply refer to $P$ as a \textit{wave operator}. We will require that $P$ is \textit{asymptotically flat}. More precisely, we first define the family of norms $$\vertiii{h}_{k}=\sum\limits_{|\alpha|\leq k} \norm{\inprod{x}^{|\alpha|}\partial^\alpha h}_{\ell^1_j L^\infty([0,T]\times A_j)},$$ where $A_j=\{\inprod{x}\approx 2^j\}$ for $j\geq 0$ denote inhomogeneous dyadic annuli, $\partial=(\partial_t,\nabla_x)$ denotes the space-time gradient, and $\inprod{x}=(1+|x|^2)^{1/2}$ denotes the Japanese bracket of $x$. The notation $A\lesssim B$ indicates that $A\leq CB$ for some $C>0,$ and the notation $A\approx B$ means that $B\lesssim A\lesssim B$. In the definition of the $A_j$'s, we require that these implicit constants are compatible to cover $\R^3.$ We may now define the $AF$ norm topology via 
$$\norm{(h,A,a,V)}_{AF}=\vertiii{h}_2+\vertiii{\inprod{x}A}_{1}+\vertiii{\inprod{x} a}_{1}+\vertiii{\inprod{x}^2 V}_0.$$

\begin{definition}
    We say that $P$ is \textit{asymptotically flat} if 
    $\norm{(g-m,A,a,V)}_{AF}<\infty$,
    and  
\begin{align*}
    \norm{\inprod{x}^{|\alpha|}\partial^\alpha g}_{\ell_j^1L^\infty([0,T]\times A_j)}&\lesssim_{\alpha} 1,\qquad |\alpha|\geq 3,\\
    \norm{\inprod{x}^{|\alpha|+1}\partial^\alpha A}_{\ell_j^1L^\infty([0,T]\times A_j)}+\norm{\inprod{x}^{|\alpha|+1}\partial^\alpha a}_{\ell_j^1L^\infty([0,T]\times A_j)}&
   \lesssim_{\alpha} 1,\qquad |\alpha|\geq 2,\\
    \norm{\inprod{x}^{|\alpha|+2}\partial^\alpha V}_{\ell_j^1L^\infty([0,T]\times A_j)}&\lesssim_{\alpha} 1,\qquad |\alpha|\geq 1,
\end{align*} 
 where $m$ denotes the Minkowski metric.
\end{definition}
The latter conditions are placed to work in standard symbol classes, as opposed to those with limited regularity. They extend the flavor of the estimate present in the definition of asymptotic flatness, but one does not require summability over the differentiation indices. 

We will primarily be interested in when $P$ is \textit{stationary}, which occurs when $(g,A,a,V)$ are independent of $t$. In particular, $\partial_t$ is a Killing field for $g$ when $P$ is stationary. When $P$ is a damped wave operator, we will \textit{always} assume that
$a$ is independent of $t$, $a$ is non-negative for $|x|\leq 2R_0$, and $a$ is positive on an open subset of $\{|x|\leq R_0\}$. For $|x|\leq 2R_0$, the damping term in $P$ behaves as typical viscous damping, and we consider it as a general time-independent first-order perturbation term for $|x|>2R_0$. Further, we will assume throughout that $\partial_t$ is uniformly time-like while the constant time slices are uniformly space-like (i.e. $dt$ is uniformly time-like). The former condition implies an ellipticity condition on the terms in $P$ which are independent of $D_t$ (see Section \ref{ee defs} for more), while the latter guarantees that $g^{00}\gtrsim -1$, allowing $P$ to be reduced to a normal form of $g^{00}=-1$.   

Next, we introduce parameters that quantify various aspects of asymptotic flatness. Namely, we instantiate \begin{itemize}
    \item the parameters $M_0$, $R_0$ and $\textbf{c}$, which are such that \begin{align*}
       \norm{(g-m,A,a,V)}_{AF}&\leq M_0,\qquad\qquad\norm{(g-m,A,a,V)}_{AF_{>R_0}}\leq \textbf{c}\ll 1, 
        \end{align*} where the subscript denotes the restriction of the norm to $\{|x|>R_0\}.$ The parameter $\textbf{c}$ should be viewed as being fixed first, after which we find a suitable $R_0$ for which the above holds. 
            \item the sequence  $(c_j)_{j\geq \log_2 R_0}$ of positive real numbers satisfying that
        \begin{align*}
      \norm{(g-m,A,a,V)}_{AF(A_j)}\lesssim c_j,\qquad\qquad
        \sum\limits_j c_j\lesssim \textbf{c}, 
    \end{align*} where $\norm{\cdot}_{AF(A_j)}$ denotes the restriction of the $AF$ norm to the dyadic region $A_j.$ We may assume without any loss of generality that the sequence is slowly-varying, i.e. $$c_j/c_k\leq 2^{\delta|k-j|},\qquad\delta\ll 1.$$
    The integers $0\leq j<\log_2 R_0$ index finitely many dyadic regions. Since the union of such regions is compact, we can extend $(c_j)$ to such indices in an arbitrary manner, although we will assume that the appended terms in the sequence are positive and allow $(c_j)$ to remain slowly-varying.
    \end{itemize}
    \begin{remark} The implicit constants in our inequalities throughout this work are permitted to (and often will) depend on $\textbf{c}$ (which is universal), $M_0$ (also universal), and $R_0$ (only depends on $\textbf{c}$). However, they cannot depend on $T$; the negation would be problematic and not allow one to e.g. take the limit as $T\rightarrow\infty$.
    \end{remark}
    
Next, we introduce definitions regarding the skewness of $P-iaD_t$ and the size of the damping relative to the magnetic potential. We will generically refer to these as \textit{symmetry conditions}.
\begin{definition}\label{symm def}
We say that $P$ is 
\begin{itemize}
    \item of \textit{symmetric wave type} if $A$ and $V$  are real-valued.
    \item strongly $\varepsilon$-damping dominant if  $$a(x)\geq \varepsilon^{-1} |\imag A(x)|,\qquad |x|\leq 2R_0.$$ 
    \item weakly $\varepsilon$-damping dominant if
   $$\frac{{g^{0j} \xi_j\pm\sqrt{\left(g^{0j}\xi_j\right)^2+g^{ij}\xi_i\xi_j}}}{\pm 2 \sqrt{\left(g^{0j}\xi_j\right)^2+g^{ij}\xi_i\xi_j}}(a+g^{0\alpha} \imag A_\alpha) \pm\frac{\xi_k}{{2 \sqrt{\left(g^{0j}\xi_j\right)^2+g^{ij}\xi_i\xi_j}}}g^{k\alpha}\imag A_\alpha \geq \varepsilon a,\qquad|x|\leq 2R_0, \ \ \xi\neq 0.$$
\end{itemize}
\end{definition}
\begin{remark} The parameter $\varepsilon$ in the definition of the strongly $\varepsilon$-damping dominant condition should be viewed as being fixed after $R_0, M_0$. The same will not be true for the weakly $\varepsilon$-damping dominant condition, which we will assume holds for some $\varepsilon>0$ independent of the previously-discussed parameters. Additionally, the first two listed conditions are stated for $|x|\leq 2R_0$, since the condition is already true in the exterior region due to the growth conditions present in the $AF$ topology.
\end{remark}
\begin{remark}\label{weak strong damp}
The weakly $\varepsilon$-damping dominant condition, while seemingly esoteric, will arise naturally in our high frequency analysis and constitutes a sharpening of the more concrete and directly-verifiable strongly $\varepsilon$-damping dominant condition. First, we remark that we will need the latter to hold for sufficiently small $\varepsilon>0$ in order to obtain our results, whereas we only need the former to hold for \textit{some} $\varepsilon>0$. The range of $\varepsilon$ which are sufficient for our results to hold for strongly $\varepsilon$-damping dominant $P$ will guarantee that the weaker condition holds (for a different $\varepsilon$). We will now discuss their relationship more explicitly. 

In the language of Section \ref{high freq section}, the weakly $\varepsilon$-damping dominant condition may be written as 
$$\frac{b^\pm}{b^\pm-b^\mp}(a+g^{0\alpha}\imag A_\alpha)+ \frac{\xi_k}{b^\pm-b^\mp }g^{k\alpha}\imag A_\alpha \geq\varepsilon a,\qquad\qquad |x|\leq 2R_0,\qquad\xi\neq 0, $$ 
where  $$b^+(x,\xi)>0>b^-(x,\xi),\qquad b^\pm(x,\xi)\approx \pm|\xi|,\qquad\xi\neq 0.$$ Hence, the weakly $\varepsilon$-damping dominant condition may be recast (up to a fixed constant coefficient on the lower-bound side) in more simple terms as $$a+g^{0\alpha}\imag A_\alpha \pm \frac{\xi_k}{|\xi|}g^{k\alpha}\imag A_\alpha \geq \varepsilon a,\qquad |x|\leq 2R_0,\qquad \xi\neq 0$$ If $P$ is strongly $\varepsilon$-damping dominant, then 
\begin{align*}a+g^{0\alpha}\imag A_\alpha\pm \frac{\xi_k}{|\xi|}g^{k\alpha}\imag A_\alpha &\geq a-(2\norm{g-m}_{L^\infty_{<R_0}}+1)|\imag A|\\
&\geq \left(1-\varepsilon(2\norm{g-m}_{L^\infty_{<R_0}}+1)\right)a,\ \ \ \ |x|\leq 2R_0,\ \ \xi\neq 0.
\end{align*}
If $\varepsilon<(2\norm{g-m}_{L^\infty_{<R_0}}+1)^{-1},$  then the weakly $\varepsilon'$-damping dominant condition holds with $$\varepsilon'=1-\varepsilon(2\norm{g-m}_{L^\infty_{<R_0}}+1)>0.$$ In particular, $0<\varepsilon'< 1$. 

On the other hand, if $g=m$, then the simplified weakly $\varepsilon$-damping dominant condition  from above stipulates that 
$$(a-\imag A_0) \pm \frac{1}{|\xi|}\sum\limits_{k=1}^3\xi_k \imag A_k\geq  \varepsilon a,\qquad |x|\leq 2R_0,\qquad \xi\neq 0.$$ We are working with the simplified form for consistency with the strong-to-weak implication, although $b^\pm(x,\xi)=\pm |\xi|$ in the flat case, so one could be more precise if desired. Written differently, we have that $$a\left(1-\varepsilon\right)\geq\imag A_0\pm\sum\limits_{k=1}^3\frac{\xi_k}{|\xi|} \imag A_k,\qquad |x|\leq 2R_0,\qquad \xi\neq 0.$$ Considering $\xi=(\imag A_1, \imag A_2, \imag A_3)$ and taking the positive sign in the inequality gives
$$a\left(1-\varepsilon\right)\geq \imag A_0+|\xi|\geq |\imag A|$$ if $\imag A_0\geq 0,$ which shows that $P$ is strongly $\varepsilon'$-damping dominant, with $\varepsilon'=1-\varepsilon>0$ for $\varepsilon<1$ (which is consistent with the prior strong-to-weak implication established earlier). Hence, the conditions are equivalent in the flat case for non-negative $\imag A_0$ and appropriate (and different) $\varepsilon
$. Once one considers more complicated metrics, even  asymptotically Euclidean metrics, the presence of potentially large metric perturbations in the interior make the establishment of a positive $\varepsilon'$ difficult and the weak-to-strong implication less clear.
\end{remark}

The $\varepsilon$-damping dominant conditions are, to our knowledge, original and seemingly more natural in the context of damped waves; we will only discuss the strong version for now since it is simpler to parse and easier to verify (as well as stronger). It stipulates that the damping must be the dominant skew-adjoint term (at the principal level) for $|x|\leq 2R_0.$ Thus, the magnetic potential is permitted to have large imaginary components for $|x|\leq 2R_0,$ so long as the damping is more significant, in the sense stated above. Moreover, the $\varepsilon$ in the strongly $\varepsilon$-damping dominant condition must satisfy that $\varepsilon\lesssim (2\norm{g-m}_{L^\infty_{<R_0}}+1)^{-1}$ for our results to hold in this work under this more strict condition (in which case the weakly $\varepsilon$-damping dominant condition holds with a different $\varepsilon$, as discussed in the previous remark). This indicates that if the metric perturbation is sufficiently large within $\{|x|\leq R_0\}$, then $|\imag A|$ must be sufficiently smaller than $a$. If the perturbation is small, then $|\imag A|$ and $a$ are permitted to be close in size for $|x|\leq 2R_0$.

One may desire consideration of a so-called $\varepsilon$\textit{-weakly magnetic}, which stipulates that $\norm{(0,\imag A,0,0)}_{AF}\leq\varepsilon.$ This is a slight weakening of the more traditional $\varepsilon$-almost symmetric condition which includes $\imag V$ (present in, e.g. \cites{MST20}), since we will primarily be concerned with the interplay between $\imag A$ and $a$. Although the condition appears similar to the strongly $\varepsilon$-damping dominant condition, neither is necessarily stronger. With this being said, these conditions arise in the high frequency analysis in our work, and the particular choice of $\varepsilon$ in the $\varepsilon$-weakly magnetic condition necessary for the results to hold directly implies the $\varepsilon$-damping dominant condition with $\varepsilon\lesssim (2\norm{g-m}_{L^\infty_{<R_0}}+1)^{-1}$ (which then implies the weak condition, as well).  

\subsection{Local Energy Spaces and Estimates}\label{ee defs}
The study of localized energy estimates dates back to the work of \cites{Mora66, Mora68,Mora75}, \cites{MRS77} on Minkowski space-time. Local energy estimates constitute a powerful measure of dispersion, implying Strichartz estimates (\cites{BT07, BT08}, \cites{JSS90, JSS91}, \cites{MMT08}, \cites{MMTT10}, \cites{MT09,MT12}, \cites{Tat08}, \cites{Toh12}) and pointwise decay estimates (e.g.  \cites{MTT12}, \cites{Tat13}, \cites{Morg20}, \cites{MW21}, \cites{Looi21}, \cites{Hi23}). The latter has been used to tackle various generalizations of Price's law which, in its simplest form, conjectured a $t^{-3}$ pointwise decay rate of waves on non-rotating black hole space-times; see the listed references on pointwise decay and \cites{Hi22} (which settles the conjecture affirmatively in full generality) for more. The particular local energy estimate of interest in this work, integrated local energy decay, is a powerful quantitative statement given in the form of an inequality which may be qualitatively interpreted as expressing that the energy of a wave must disperse quickly enough to be time-integrable within any compact region of space.
    
    In order to define the relevant energy inequalities explicitly, we will first define the local energy norms \begin{align*}\norm{u}_{LE}=\sup_{j\geq 0}\norm{\langle x\rangle^{-1/2} u}_{L^2L^2\big(\R_+\times A_j\big)},\qquad\qquad
\norm{u}_{LE^1}=\norm{\partial u}_{LE}+\norm{\langle x\rangle^{-1}u}_{LE}.
\end{align*}  The predual norm to the $LE$ norm is the $LE^*$ norm, which is defined as $$ \norm{f}_{LE^*}=\sum\limits_{j=0}^\infty\norm{\langle x\rangle^{1/2} f}_{L^2L^2\big(\R_+\times A_j\big)}.$$ 
If we wish for the time interval to be e.g. $[0,T]$ in the above norms, then we will use the notation e.g. $\norm{u}_{LE[0,T]}$. 
A subscript of $c$ on any of these spaces denotes compact spatial support, whereas a subscript of $0$ denotes the closure of $C_c^\infty$ in the relevant norm. 

Analogous to the local energy spaces, we define the spatial local energy spaces $\mathcal{LE}, \mathcal{LE}^1, \mathcal{LE}^*$ when the time variable is fixed (and there is no time derivative nor integral involved in the norms, either). 
We will also require spaces which allow us to track dependence on a spectral parameter $\omega$, namely
\begin{align*}\mathcal{LE}_\omega^1=\mathcal{LE}^1\cap |\omega|^{-1}\mathcal{LE},\qquad\qquad
\dot{H}^1_\omega=\dot{H}^1\cap|\omega|^{-1}L^2.
\end{align*} These spaces are equipped with the norms 
\begin{align*}
    \norm{u}_{\mathcal{LE}_\omega^1}=\norm{u}_{\mathcal{LE}^1}+|\omega|\norm{u}_{\mathcal{LE}},\qquad\qquad
    \norm{u}_{\dot{H}_\omega^1}=\norm{u}_{\dot{H}^1}+|\omega|\norm{u}_{L^2},
\end{align*} respectively. They will become relevant in this work when we introduce the resolvent formalism in Section \ref{resolvent section}.
Now, we will define local energy decay.
\begin{definition}
We say that \textit{(integrated) local energy decay} holds for an asymptotically flat damped wave operator if \begin{equation}\label{LED}
\norm{u}_{LE^1[0,T]}+\norm{\partial u}_{L^\infty L^2[0,T]}\lesssim\norm{\partial u(0)}_{L^2}+\norm{Pu}_{LE^*+L^1L^2[0,T]},\end{equation} with the implicit constant being independent of $T$.
\end{definition}

The next estimate is a weaker variant of local energy decay which does not see complex eigenvalues nor non-zero real resonances (to be defined in Section \ref{resolvent section}).
\begin{definition} 
We say that \textit{two-point local energy decay} holds for an asymptotically flat damped wave operator if \begin{equation}\label{two point LED}
\norm{u}_{LE^1[0,T]}+\norm{\partial u}_{L^\infty L^2[0,T]}\lesssim\norm{\partial u(0)}_{L^2}+\norm{\partial u(T)}_{L^2}+\norm{Pu}_{LE^*+L^1L^2[0,T]},\end{equation} with the implicit constant being independent of $T$.
\end{definition}

The last estimate is a standard uniform energy bound, which embodies the behavior that one observes when the energy of the system is non-increasing.
\begin{definition}
We say that \textit{uniform energy bounds} hold for an asymptotically flat damped wave operator if
\begin{equation}\label{unif energy}
 \norm{\partial u}_{L^\infty L^2[0,T]}\lesssim\norm{\partial u(0)}_{L^2}+\norm{Pu}_{L^1L^2[0,T]},\end{equation} with the implicit constant being independent of $T$.   
\end{definition}
Notice that the uniform energy bounds provide a link between the two-point local energy decay estimate and the local energy decay estimate. In order to discuss the uniform estimate further, let $Pu=f$, and assume that $P$ is stationary, asymptotically flat, and of symmetric wave type. Consider the sesquilinear map $\boldsymbol{E}$ on the energy space $\mathcal{E}:=\dot{H}^1\times L^2$ defined by $$\boldsymbol{E}[\boldsymbol{u},\boldsymbol{v}](t)=\int\limits_{\R^3}P_0 u\bar{v}-g^{00}\partial_t u\overline{\partial_t v}\, dx,\qquad P_0=P\big|_{D_t=0},\qquad \boldsymbol{w}=(w,\partial_t w)\in\mathcal{E}.$$ 
This induces a quadratic energy functional, i.e. an energy form $$E[\boldsymbol{u}](t):=\boldsymbol{E}[\boldsymbol{u},\boldsymbol{u}](t).$$ When $\partial_t$ is a uniformly time-like vector field (which we will assume throughout), $P_0$ is guaranteed to be elliptic (in the principal sense). 
Direct integration by parts gives that $$\frac{d}{dt}E[\boldsymbol{u}](t)=2\real \int\limits_{\R^3}\partial_t u\bar{f}\, dx-2\real\int\limits_{\R^3}a|\partial_tu|^2\, dx.$$ 
If $a$ is real-valued and non-negative, then it follows that $$E[\boldsymbol{u}](t)\lesssim E[\boldsymbol{u}](0)+\int\limits_0^T\int\limits_{\R^3}|f \partial_t u|\, dxdt
,\qquad 0\leq t\leq T.$$ 
If $A,V\equiv 0$, then $P_0$ is \textit{uniformly} elliptic (in the sense of the \textit{full symbol}), which implies that the energy is coercive, i.e. $$E[\boldsymbol{u}](t)\approx\norm{\partial u(t)}_{L^2}^2.$$ 
Energy coercivity applied to the above now provides the uniform energy bound (\ref{unif energy}). 

When $A$ and $V$ are non-zero, one is not guaranteed energy coercivity (and thus not guaranteed uniform energy bounds), even if $a$ is non-negative. However, one does get an \textit{almost}-coercive energy statement from the ellipticity of $P_0$ and the asymptotic flatness assumption: \begin{equation}\label{almost coerc}\norm{\partial u(t)}_{L^2}^2\lesssim E[\boldsymbol{u}](t)+\norm{u(t)}_{L^2_c}^2.\end{equation} 
When $P$ is not of symmetric wave type, then one  redefines the energy by replacing $P_0$ with its symmetric part (see Section \ref{dichot section}); the estimate (\ref{almost coerc}) still holds for the energy associated to $P_0$ by ellipticity regardless of the symmetry. The uniform bound (\ref{unif energy}) still need not hold when $P_0$ is replaced by its symmetric part. Thus, it is not straightforward to transition from (\ref{two point LED}) to (\ref{LED}) when $A,V\not\equiv 0$, even for well-signed dampings;  this is true even in the symmetric wave type case. 

\subsection{Past Results}

In \cites{MST20}, the authors proved (amongst other results) that local energy decay for stationary $AF$ wave operators is equivalent to an absence of \textit{geodesic trapping}, \textit{negative eigenfunctions}, and \textit{real resonances}. We will describe trapping in Section \ref{high freq section} and the spectral objects in Section \ref{resolvent section}. In short, trapping occurs when there exist bicharacteristic rays which live within a compact set for all time. The spectral obstructions correspond to singular behavior of the \textit{resolvent} - negative eigenfunctions live in $L^2$ and have corresponding eigenvalue in the lower half-plane, whereas real resonances lie on the real line and have a corresponding resonant state which lives in a local energy space (one must distinguish between zero and non-zero resonances). In \cites{MST20}, the authors also did not necessarily possess a coercive energy; they proved local energy decay by establishing (\ref{two point LED}) using local energy estimates in different frequency regimes, then they utilized resolvent estimates.

 The work \cites{BR14} utilized dissipative Mourre commutator methods to establish that if the space-time is stationary and asymptotically \textit{Euclidean} (i.e. $(\R^4,g)$ is a product manifold and hence possesses no non-trivial metric cross terms $dt\otimes dx^j$), then one has local energy decay for $AF$ stationary damped wave operators with $A,V\equiv 0$ and $a$ being a non-negative short-range potential, provided that $a$ satisfied a dynamical hypothesis called \textit{geometric control}. This condition requires that all trapped null bicharacteristic rays intersect where $a>0$ (see Definition \ref{GCC} for a precise definition), although the authors of \cites{BR14} only required this for trapped geodesics due to the product manifold structure. Geometric control dates back to \cites{RT74}, which utilized it to obtain exponential energy decay (i.e. \textit{uniform stabilization}) for dissipative problems on compact manifolds. In \cites{Kof22}, we generalized the work of \cites{BR14} to the asymptotically flat case (i.e. allowed the metric to possess cross terms). To be precise, we will record this result, which is Theorem 1.9 in \cites{Kof22} (adding in the missing assumption of uniformly space-like time slices).
 
 \begin{Th}
Let $P$ be a stationary, asymptotically flat damped wave operator satisfying the geometric control condition with $A,V\equiv 0$ and $\supp a\subset\{|x|\leq R_0\}$, and suppose that $\partial_t$ is uniformly time-like while the constant time slices are uniformly space-like. Then, local energy decay holds, with the implicit constant in (\ref{LED}) independent of $T$.
\end{Th}

In \cites{Kof22}, uniform energy bounds held as a result of the conditions on the damping and lack of potentials, which made it sufficient to prove the two-point bound in order to establish local energy decay. In order to establish the former estimate, \cites{Kof22} followed the strategy set forth in \cites{MST20}:
\begin{enumerate}
    \item Establish local energy estimates that imply local energy decay for Schwartz functions, whose corresponding function space we denote as $\mathcal{S}$, which are cut off to high, medium, and low frequency regimes.
    \item Utilize a time frequency partition of unity to prove the estimate $$\norm{u}_{LE^1}\lesssim\norm{Pu}_{LE^*},\qquad u\in\mathcal{S}.$$
    \item Apply an extension procedure to add back in the energy at times $0$ and $T$. 
\end{enumerate} 
\subsection{Statement of Present Results}\label{results section} We will generalize the work of \cites{Kof22} to include the lower-order magnetic and scalar potentials, along with more general damping functions. Our first result is an extension of the high frequency estimate \cites{Kof22} to our setting. This  high frequency estimate first arose in \cites{MST20} for non-damped waves.
\begin{Th}\label{high freq thm}
Let $P$ be a stationary, asymptotically flat damped wave operator which satisfies the geometric control  and weakly $\varepsilon$-damping dominant conditions for some $\varepsilon>0$. Additionally, assume that $\partial_t$ uniformly time-like while the constant time slices are uniformly space-like. Then, the high frequency estimate \begin{equation}\label{high freq est}\norm{u}_{LE^1[0,T]}+\norm{\partial u}_{L^\infty L^2[0,T]}\lesssim\norm{\partial u(0)}_{L^2}+\norm{\langle x\rangle^{-2}u}_{LE[0,T]}+\norm{Pu}_{L^1L^2+LE^*[0,T]}\end{equation} holds with an implicit constant which is independent of $T$.
\end{Th}

Notice that if we are in the setting of \cites{Kof22} then $A,V, a\chi_{|x|>R_0}\equiv 0$, which implies the weakly $\varepsilon$-damping dominant condition holds. Hence, this theorem is a strengthening of the corresponding result in \cites{Kof22}.

This theorem is one of the primary results whose proof must be adapted from \cites{Kof22} to account for the lower-order terms. It is also where one requires the most substantial deviation from \cites{MST20}, since trapping is high frequency. We utilize the symmetry-based assumptions in the theorem so that we may deal with the additional lower-order terms. In particular, this is where the conditions on the interaction between the damping and magnetic potentials come into play.  In order to leverage the sign of the damping to mitigate the harmful effects of the trapping, we must limit the magnetic potential appropriately.

The medium and low frequencies are not affected by the damping (which may simply be viewed as a general sub-principal $AF$ term, as opposed to being leveraged like in the high frequency setting) and follow directly from the work in \cites{MST20}. This allows us to establish the two-point local energy estimate under the hypothesis that zero is not a resonance, which is needed in the low frequency regime.
\begin{Th}\label{two point thm}
Let $P$ be a stationary, asymptotically flat damped wave operator which satisfies the zero non-resonance, geometric control, and weakly $\varepsilon$-damping dominant conditions for some $\varepsilon>0$. Additionally, assume that $\partial_t$ is uniformly time-like while the constant time slices are uniformly space-like. Then, two-point local energy decay holds, with the implicit constant in (\ref{two point LED}) independent of $T$.
\end{Th}
We will define the zero non-resonance condition in Section \ref{rem freq section} and its relation to zero resonant states in Section \ref{resolvent section}. If we additionally impose that $P$ satisfies the $\varepsilon$-damping dominant condition with $|\Im A|$ replaced by $|A|, |\nabla A|,$ and $|V|$, then we obtain a straightforward energy dichotomy (just as in \cites{MST20}, where it is Theorem 2.16) as a consequence of the two-point local energy estimate and a uniform energy relation. The given symmetry conditions fulfill a similar role to a condition on the absence of non-zero embedded resonances, and it is needed to obtain the aforementioned uniform energy relation (i.e. an almost-conserved energy property if $Pu=0$) which will appear in the proof. Recall the definition of $\mathcal{E}$ given in Section \ref{ee defs}.
 \begin{Th}\label{energy dich}
 Let $P$ be a stationary, asymptotically flat damped wave operator which is strongly $\varepsilon$-damping dominant with respect to the magnetic potential, the gradient of the magnetic potential, and the scalar potential for sufficiently small $\varepsilon\ll_{R_0, M_0} 1$, equivalently $$a(x)\geq\varepsilon^{-1}(|A(x)|+|\nabla A(x)|+|V(x)|),\qquad |x|\leq 2R_0,$$
 and satisfies the zero non-resonance and geometric control conditions. Additionally, assume that $\partial_t$ is uniformly time-like while the constant time slices are uniformly space-like.
 Then, there exists an $\alpha>0$ so that any solution to $$Pu=f,\qquad u[0]\in\mathcal{E},\qquad f\in LE^*+L^1_tL^2_x$$ satisfies one of the following two properties:
 \begin{enumerate}
 \item Exponential growth asymptotics in terms of the data and forcing: $$\norm{\partial u(t)}_{L^2}\gtrsim e^{\alpha t}\left(\norm{\partial u(0)}_{L^2}+\norm{f}_{LE^*+L^1L^2[0,\infty)}\right),\qquad t\gg 1.$$
     \item Local energy decay: $$\norm{u}_{LE^1[0,\infty)}+\norm{\partial u}_{L^\infty L^2[0,\infty)}\lesssim\norm{\partial u(0)}_{L^2}+\norm{f}_{LE^*+L^1L^2[0,\infty)}.$$ 
 \end{enumerate}
 \end{Th}
  \begin{remark}
    The condition on the damping can be readily weakened to assuming that $P$ is weakly $\varepsilon$-damping dominant (in order for Theorem \ref{two point thm} to apply) and that, for $\varepsilon$ small enough and all $t>0$, the estimate
\begin{align*}\int\limits_0^t\int\limits_{\R^3} \partial_s u\overline{P^au}\, dxds
   =C(\varepsilon)\int\limits_0^t\int\limits_{B_{2R_0}(0)} a|\partial_s u|^2\, dxds \nonumber +\mathcal{O}(\max(\varepsilon^k, \textbf{c}))\norm{u}_{LE^1}^2
  \end{align*} holds for some $C(\varepsilon)>0$ and $k\in (0,1)$, where $P^a$ is the the skew-adjoint part of $P$. The assumption given in the dichotomy was stated as such merely for tractability, as it will both satisfy the relevant hypothesis in Theorem \ref{two point thm} and allows the above bound to hold as a consequence of H\"older's inequality, Young's inequality for products, and asymptotic flatness.
  \end{remark}

  At least at a heuristic level, the solutions which exhibit case (1) behavior stem from eigenvalues in the lower half-plane of the corresponding stationary problem (see Section \ref{resolvent section}), which represent poles of the resolvent.  The resolvent has meromorphic continuation to the entire lower half-plane, and the poles must occur within a relatively compact subset of frequencies. In particular, there are only finitely many such eigenvalues, and each generalized eigenspace has finite dimension by Fredholm theory. We also remark that we do not obtain improvements from the corresponding result in \cites{MST20} here, nor do we obtain versions of their non-stationary results since our high frequency work exploits the stationarity. While more refined energy space decompositions are likely (and are also present in \cites{MST20} within the non-damped setting), the proper statements and results in the context of damped waves are not clear to us at this time.

Finally, we have local energy decay. Here, we must further assume that $P$ satisfies various spectral hypotheses, which are defined in Section \ref{resolvent section}.
\begin{Th}\label{LED thm}
Let $P$ be a stationary, asymptotically flat damped wave operator which satisfies the zero non-resonance, geometric control, and weakly $\varepsilon$-damping dominant conditions for some $\varepsilon>0$.  Suppose further that $P$ has no negative eigenfunctions nor real resonances and that $\partial_t$ is uniformly time-like while the constant time slices are uniformly space-like. Then, local energy decay
holds, with the implicit constant in (\ref{LED}) independent of $T$.
\end{Th}

In order to prove this result, we cannot necessarily rely on uniform energy bounds to pass from the two-point local energy estimate to local energy decay like in \cites{Kof22} and instead rely on spectral theory as in \cites{MST20}. The structure of proving Theorem \ref{LED thm}, and hence of the overall paper, is as follows and is motivated by \cites{MST20}:
\begin{enumerate}
    \item Section \ref{high freq section}: Extend the high frequency analysis from \cites{Kof22} to the present context. One must take care to consider how the damping interacts with the remaining principal skew-adjoint terms; this is the purpose of the $\varepsilon$-damping dominant conditions. We will explain the results that we borrow and why they apply here.
    \item Section \ref{rem freq section}: Provide an overview of the low and medium frequency analysis; these do not require change from \cites{MST20}, since the damping is simply treated as a first-order $AF$ perturbation term. From here, Theorem \ref{two point thm} follows readily from the work in \cites{MST20}, \cites{Kof22}. As an immediate consequence, we will establish Theorem \ref{energy dich} in Section \ref{dichot section}, although this has no bearing on the proof of Theorem \ref{LED thm}.
    \item  Section \ref{resolvent section}:  Summarize the necessary resolvent theory in \cites{MST20} required to prove Theorem \ref{LED thm}. This will require Theorem \ref{high freq thm} and Theorem \ref{two point thm}. Once one is armed with the relevant frequency estimates, the work in \cites{MST20} applies rather directly. We will summarize and/or provide many (but not all) of their arguments for the required results, especially where we believe that further elucidation would be beneficial for the sake of exposition. We require little deviation from their theory in our present work.
\end{enumerate}

\begin{remark}
As a consequence of Remark \ref{weak strong damp}, Theorems \ref{high freq thm}, \ref{two point thm}, and \ref{LED thm} hold if the weakly $\varepsilon$-damping dominant condition is replaced by the strongly $\varepsilon$-damping dominant condition with $\varepsilon\lesssim(2\norm{g-m}_{L^\infty}+1)^{-1}$.
\end{remark}
\begin{remark}
As opposed to writing the d'Alembertian in divergence form (i.e. $\Box_g=D_\alpha g^{\alpha\beta}D_\beta$) and utilizing the volume form $dV=dxdt$ for our analysis, one  could work with the geometric d'Alembertian (that is, the Laplace-Beltrami form) $$\widetilde{\Box}_g=|g|^{-1/2}D_\alpha |g|^{1/2}g^{\alpha\beta}D_\beta,\qquad |g|=|\det g^{\alpha\beta}|$$ with the volume form $dV=|g|^{1/2}dxdt.$ %
Each d'Alembertian is symmetric with respect to the associated volume form on $L^2(dV)$. One can transition from the latter framework to the former by conjugating the operator by $|g|^{1/4}$ (see e.g. \cites{Tat13}, \cites{Morg20}); lower-order terms arise, but they are permissible in view of the magnetic and scalar potential terms already allowable in $P$. For this reason, we are working with the former case.
\end{remark}
\noindent\textbf{Acknowledgments.}
The author would like to thank Jason Metcalfe for helpful discussions. He would also like to express his gratitude to the anonymous reviewers for their feedback.
\vskip .1in
\noindent\textbf{Declarations.}
The author reports there are no competing interests to declare.
\section{Frequency Analyses and Two-Point Local Energy Decay}
In this section, we will prove Theorem \ref{two point thm} using high, low, and medium frequency analyses. The low and medium frequency work follows directly from that in \cites{MST20}, while the high frequency work is a variation on \cites{Kof22} and requires the symmetry assumptions as described in Section \ref{Background}.

To start, we will define cutoff notation which we will use throughout the duration of the paper. Namely, we will fix
 $\chi\in C_c^\infty\text{ non-increasing},\ \chi\equiv 1\text{ for } |x|\leq 1,\ \chi\equiv 0\text{ for } |x|> 2$ and define  $\chi_{< R}(|x|)=\chi(|x|/R), $ $  \chi_{>R}=1-\chi_{< R}.$ We will assume further that $\chi$ is the square of a smooth function for notational convenience (as otherwise, we could work with $\chi^2$).
We will occasionally add the variable into the subscript to make the specific dependence clear (e.g. write $\chi_{|\xi|<\lambda}$ when working truncating the spatial frequencies $\xi$ below a threshold $\lambda$).
\subsection{High Frequency Analysis}\label{high freq section}
We will recall the relevant framework and results from \cites{Kof22} needed to prove the high frequency estimate. When deviation occurs, we will proceed carefully and explicitly. The need for the $\varepsilon$-damping dominant condition only comes up in one place, in the proof of Lemma \ref{escape}. 
Throughout this section, we will assume that $P$ is stationary. This work is motivated by \cites{Kof22}, \cites{MST20}, and \cites{BR14}.

The high frequency analysis is rooted in the behavior of the bicharacteristic flow generated by the principal symbol of $P$. First, we make a minor simplification. Since the constant time slices are assumed uniformly space-like, it follows that $g^{00}\lesssim -1$. Dividing through by $-g^{00}$ preserves the assumptions on the operator coefficients (see e.g. \cites{MT12}); hence, we may assume that $g^{00}=-1$.  

After these modifications, the principal symbol of $P$ is the smooth function $$p(x,\tau, \xi)=-(\tau^2-2\tau g^{0j}(x)\xi_j-g^{ij}(x)\xi_i\xi_j),\qquad (t,x,\tau, \xi)\in T^*\R^4\setminus o.$$ Notice that since $P$ is stationary, $p$ is independent of $t$. This symbol generates a bicharacteristic flow on $\R\times T^*\R^4$ given by  $\varphi_s(w)=\left(t_s(w),x_s (w),\tau_s(w),\xi_s(w)\right)$
which solves  
$$\left\{
\arraycolsep=1pt
\def\arraystretch{1.2}
\begin{array}{rlrl}
\dot{t}_s&=\partial_\tau p(\varphi_s(w)),& \qquad
\dot{\tau}_s&=-\partial_t p(\varphi_s(w)),\\
\dot{x}_s&=\nabla_\xi p(\varphi_s(w)),& \qquad
\dot{\xi}_s&=-\nabla_x p(\varphi_s(w)),
\end{array}\right.$$
with initial data $w\in T^*\R^4.$ Since $g$ is smooth and asymptotically flat, and $\partial_t$ is uniformly time-like, we have a unique, smooth, globally-defined flow with smooth dependence on the data.  We will have particular interest in \textit{null} bicharacteristics, i.e. those with initial data lying in the zero set of $p$, denoted $\operatorname{Char}(P)$. We remark that there is no distinguishing between $\pm p$ on $\operatorname{Char}(P)$, hence the minus sign in front of $p$ is somewhat inconsequential on it.

Using the flow $\varphi_s$, we define the \textit{forward} and \textit{backward trapped} and \textit{non-trapped} sets with respect to $\varphi_s$, respectively, as
\begin{align*}
\Gamma_{tr}&=\left\lbrace w\in T^*\R^4\setminus o:
\sup_{s\geq 0} |x_{ s}(w)|<\infty\right\rbrace \cap \operatorname{Char}(P),\\
\Lambda_{tr}&=\left\lbrace w\in T^*\R^4\setminus o:
\sup_{s\geq 0} |x_{- s}(w)|<\infty\right\rbrace \cap \operatorname{Char}(P),\\
\Gamma_{\infty}&=\left\lbrace w\in T^*\R^4\setminus o:
 |x_{s}(w)|\rightarrow\infty\operatorname{ as }s\rightarrow\infty\right\rbrace \cap \operatorname{Char}(P),\\
\Lambda_{\infty}&=\left\lbrace w\in T^*\R^4\setminus o:
 |x_{-s}(w)|\rightarrow\infty \operatorname{ as }s\rightarrow\infty\right\rbrace \cap \operatorname{Char}(P).
\end{align*} 
The \textit{trapped} and \textit{non-trapped} sets are defined as 
$$\Omega^p_{tr}=\Gamma_{tr}\cap\Lambda_{tr}\qand\Omega^p_{\infty}=\Gamma_{\infty}\cap\Lambda_{\infty},$$ respectively. The flow is said to be  \textit{non-trapping} if $\Omega^p_{tr}{\,=\,}\emptyset$. Otherwise, the flow is said to possess trapping.
Now, we may precisely state the geometric control condition.
\begin{definition}
    We say that the geometric control condition holds if \begin{align}\label{GCC}
    (\forall w\in \Omega^p_{tr})(\exists s\in\R)\ \ a(x_s(w))>0.
    \end{align}
\end{definition}
In order to leverage the sign of the damping, \cites{Kof22} utilized a scaling argument which we will also exploit. Namely, if $u$ solves $Pu=f$, we consider  $$\tilde{u}(t,x):=u(\gamma t,\gamma x),\qquad\gamma>0.$$ For this discussion, a tilde over a function will denote dilation by $\gamma$ in each coordinate, as done to define $\tilde{u}$. If we call $$\tilde{P}=(D_\alpha+\gamma\tilde{A}_\alpha) \tilde{g}^{\alpha\beta}(D_\beta+\gamma\tilde{A}_\beta)+i\gamma \tilde{a} D_t+\gamma^2 \tilde{V},$$ then  $$\tilde{P} (\gamma^{-2}\tilde{u})=\tilde{f}\qquad\text{ if and only if }\qquad  Pu=f.$$ The benefit of this scaling is that we obtain an arbitrarily large constant in front of the damping function. However, we underscore that such a large constant is also inherited by the magnetic potential.

Analogous Hamiltonian systems and trapped/non-trapped sets exist for the principal symbol $\tilde{p}$ of $\tilde{P}$, which amounts to simply dilating the coordinates of $g$. If we assume that geometric control holds for the flow generated by $p$, then it is proven in \cites{Kof22} that it holds for the flow generated by $\tilde{p}$. Since the proof is straightforward, we will omit it here and only record the result, which is Proposition 2.6 in \cites{Kof22}.
\begin{prop}
Assume that the geometric control condition (\ref{GCC}) holds. Then, for any $\gamma>0$, (\ref{GCC}) holds for the flow generated by $\tilde{p}$, with $a$ replaced by $\tilde{a}$.
\end{prop} 
Since the lower-order terms $A$ and $V$ are not at the principal level, they do not affect $p$ nor $\tilde{p}$. Henceforth, we will fix a large $\gamma>0$ and study the problem from the scaled perspective while reverting back to our original notation (e.g. no tildes).  It is readily seen that it is equivalent to prove Theorem \ref{high freq thm} for the scaled problem, where we now have a large constant in front of the damping term. 

The proof of the version of Theorem \ref{high freq thm} present in \cites{Kof22} (i.e. $A,V\equiv 0,\ \supp a\subset \{|x|\leq R_0\}$) is a positive commutator argument. At the symbolic  level, this requires the construction of an escape function and a lower-order correction term. Let $p$ and $s_{skew}$ represent the principal symbols of the self and skew-adjoint parts of $P$, respectively. Namely,
\begin{align*}
    {p}(x,\tau,\xi)&=-(\tau^2-2\tau {g}^{0j}(x)\xi_j-{g}^{ij}(x)\xi_i\xi_j)\\
    {s}_{skew}(x,\tau,\xi)&=i\gamma\left(\imag A_\alpha(x) (g^{\alpha 0}(x)\tau+g^{\alpha k}(x)\xi_k)+ \tau \real a(x)\right).
\end{align*} The multiplication by $\gamma$ in $s_{skew}$ will prove advantageous for a bootstrapping argument, which is precisely why we implement the $\gamma$-scaling. However, the imaginary part of $A$ has interaction with the damping and also features multiplication by $\gamma$. The $\varepsilon$-damping dominant condition is applied in order to retain the positivity effects of the damping. 

The precise escape function construction is as follows.
\begin{Lemma}\label{escape}
For all $\lambda>1$, there exist symbols $q_j\in S^j(T^*\R^3)$ and $m\in S^0(T^*\R^3)$, all supported in $|\xi|\geq\lambda$, so that $$ H_{{p}}q-2i{s}_{skew}q+{p}m\gtrsim \chi_{|\xi|>\lambda}\inprod{x}^{-2}\left(\tau^2+|\xi|^2\right),$$ where $q=\tau q_0+q_1$. \end{Lemma}

Here, $H_p$ is the Hamiltonian vector field induced by $p$, and $S^m(T^*\R^n)$ denotes the standard Kohn-Nirenberg symbol class of order $m$. 
To each symbol $b\in S^m(T^*\R^n)$, we have the associated \textit{Weyl quantization} of $b$, denoted $b^{\operatorname{w}}$, which is a pseudodifferential operator of order $m$ defined by the action
$$b^{\operatorname{w}}(x,D)u(x)=(2\pi)^{-3}\int\limits_{\R^n}\int\limits_{\R^n}e^{i(x-y)\cdot\xi}\ b\left(\frac{x+y}{2},\xi\right)u(y)\, dyd\xi,\qquad u\in\mathcal{S}(\R^n).$$ We will use $\Psi^m(\R^n)$ to denote the space of pseudodifferential operators on $\R^n$ of order $m$, and write $$\Psi^{-\infty}(\R^n):=\bigcap\limits_{m\in\R}\Psi^m(\R^n)$$ for the space of smoothing operators on $\R^n.$ 

Lemma \ref{escape} was proven in \cites{Kof22} in the special case of $A,V\equiv 0$ and $\supp a\subseteq\{|x|\leq R_0\}$, leading to a simplified $s_{skew}$. As in the aforementioned work, we will work with the half-wave decomposition; it is proven in \cites{Kof22} that the null bicharacterstics are equivalent through a reparameterization argument, and this fact continues to hold here without any change (lower-order terms do not affect $p$ and hence will not affect its induced bicharacteristic flow). That is, we factor $p$ as
$$p(\tau, x,\xi)=-(\tau-b^+(x,\xi))(\tau-b^-(x,\xi)),$$
where $$b^{\pm}(x,\xi)={g^{0j}(x) \xi_j\pm\sqrt{[g^{0j}(x)\xi_j]^2+g^{ij}(x)\xi_i\xi_j}}.$$ Using that $\partial_t$ is uniformly time-like, it is readily seen that $b^\pm$ are both positively homogeneous of degree 1 in $\xi$, and $$b^+(x,\xi)>0>b^-(x,\xi)$$ whenever $\xi\neq 0$. The Hamiltonians $p^\pm:=\tau-b^\pm$ also generate bicharacteristic flows  
$$\varphi_s^\pm(w){\,=\,}\left(t_s^\pm(w), x_s^\pm (w),\tau_s^\pm(w),\xi_s^\pm(w)\right)$$ on $\R\times T^*\R^4$ 
which solve the Hamiltonian systems 
$$\left\{
\arraycolsep=1pt
\def\arraystretch{1.2}
\begin{array}{rlrl}
\dot{t}_s^\pm&=1, &
\dot{\tau}_s^\pm&=0,\\
\dot{x}_s^\pm&=-\nabla_\xi b^{\pm}(\varphi_s^\pm(w)),\qquad &
\dot{\xi}_s^\pm&= \nabla_x b^{\pm}(\varphi_s^\pm(w)),
\end{array}\right.$$  with initial data $w\in T^*\R^4.$
Observe that the $(t,\tau)$ and $(x,\xi)$ systems are decoupled, allowing us to project onto the $(x,\xi)$ components of the flow without losing information. Notice that, after we project, we are no longer looking at null bicharacteristics but, rather, bicharacteristics with initial data having non-zero $\xi$ component. 

Now, we may define the forward and backward (denoted by the $\pm$ notation) trapped and non-trapped sets for the half-wave flows as
\begin{align*}
\Gamma_{tr}^{\pm}&=\left\lbrace w\in T^*\R^3\setminus o:
\sup_{s\geq 0} |x^\pm_{ s}(w)|<\infty\right\rbrace,\\
\Lambda_{tr}^{\pm}&=\left\lbrace w\in T^*\R^3\setminus o:
\sup_{s\geq 0} |x^\pm_{- s}( w)|<\infty\right\rbrace,\\
\Gamma_{\infty}^{\pm}&=\left\lbrace w\in T^*\R^3\setminus o:
|x^\pm_{s}(w)|\rightarrow\infty\ \operatorname{ as }\ s\rightarrow\infty\right\rbrace,\\
\Lambda_{\infty}^{\pm}&=\left\lbrace w\in T^*\R^3\setminus o:
|x^\pm_{-s}(w)|\rightarrow\infty\ \operatorname{ as }\ s\rightarrow\infty\right\rbrace.
\end{align*}
The trapped and non-trapped sets are
\begin{align*}
    {\Omega}_{tr}^{\pm}=\Gamma_{tr}^\pm\cap \Lambda_{tr}^\pm&\qand  {\Omega}_{\infty}^\pm=\Gamma_{\infty}^\pm\cap \Lambda_{\infty}^\pm,\\
        {\Omega}_{tr}={\Omega}_{tr}^+\cup {\Omega}_{tr}^-
   &\qand
    {\Omega}_{\infty}={\Omega}_{\infty}^+\cup {\Omega}_{\infty}^-.
\end{align*}
As a consequence of the factoring, we have the identities
$$   {\Omega}_{tr}=\Pi_{x,\xi}(\Omega^p_{tr})\qand {\Omega}_{\infty}=\Pi_{x,\xi}(\Omega^p_{\infty}),$$ where $\Pi_{x,\xi}(t,x,\tau,\xi)=(x,\xi).$  Additionally, we may re-state geometric control in terms of the factored flow.  If $w\in\Omega_{tr}$, then it is either trapped with respect the flow generated by $p^+$ or $p^-$. If it is trapped with respect to $p^+,$ then there is a time so that $w$ is flowed along a $p^+$-bicharacteristic ray to a point where the damping is positive, and similarly if it is trapped with respect to $p^-$.

Since $A$ and $V$ do not occur at the principal level for $P$, they will not affect the individual components of the escape function construction. Hence, the construction in \cites{Kof22} applies directly. The methodology in \cites{Kof22}, motivated by a combination of \cites{BR14}, \cites{MST20}, and \cites{MMT08}, is performed in the following steps:
\begin{enumerate}
    \item \textbf{On the characteristic set.} We will refer to $\{|x|\leq R_0\}$ as the \textit{interior} and $\{|x|>R_0\}$ as the \textit{exterior}.
    \begin{enumerate}
        \item \textbf{Interior, semi-bounded null bicharacteristics.} Here, one considers  semi-bounded null bicharacteristics with initial data living in the interior region. Working with semi-bounded trajectories is favorable since they include both trajectories that are trapped and those which escape slowly. Additionally, geometric control extends to such trajectories. This is where geometric control is used. 
        \item \textbf{The remainder of the interior region.} In this region, all of the trajectories escape both forward and backward in time. This region is more classical, but care must be taken both to avoid the trapping and incorporate the half-wave structure.
        \item \textbf{The exterior region.} As a consequence of asymptotic flatness, there are no trapped trajectories in this region. Here, one appeals to geometrically-adapted flat wave theory. The multiplier also allows for the absorption of an error term which arises in the prior region.
    \end{enumerate}
    \item \textbf{On the elliptic set.} Here, one requires a lower-order symbol which provides no contribution on the characteristic set and provides positivity off of it. This essentially follows from the minimization of an appropriate quadratic in the dual time variable $\tau$.
\end{enumerate}

Now, we cite the specific results from \cites{Kof22} (namely, Lemmas 2.13 and 2.16, respectively). First, we define the \textit{interior, semi-trapped set} $$\Omega_{R_0}^\pm :=\left(\Gamma_{tr}^\pm\cup\Lambda_{tr}^\pm\right) \cap\{|x|\leq R_0\}.$$ This is related to step (1a). Next, we define the function $$\Phi^\pm(x,\xi)=\left(x,\frac{\xi}{|b^\pm(x,\xi)|}\right).$$ It can be shown directly that
$$\left\{
\begin{aligned}
    x^\pm_{s}(x,\xi)&=x^\pm_s(x,\lambda \xi),\\
   \lambda \xi^\pm_{ s}(x,\xi)&=\xi^\pm_s(x,\lambda \xi)
\end{aligned}\right.$$
for any $\lambda>0$  due to the homogeneity of $b^\pm$. Since $|\xi/b^\pm(x,\xi)|\approx 1$ for $\xi\neq 0$ and $b^\pm$ is a constant of motion for the flow generated by $p^\pm$ (which explicitly utilizes that $g$ is stationary), the function $\Phi^\pm$ provides a lifting which is useful to pair with scaling arguments. The first portion of the construction (1a) is contained in the following lemma, which was motivated by \cites{BR14}. This is where geometric control is utilized. (More precisely, geometric control also applies to semi-trapped trajectories; see \cites{Kof22}.) 
\begin{Lemma}[Semi-bounded Escape Function Construction]\label{semibdd}
There exist $q_1^{\pm}\in C^\infty(T^*\R^3\setminus o)$, an open set $V_{R_0}^{\pm}\supset \Omega_{R_0}^\pm$, and $C^\pm\in\R_+$ so that $$H_{p^\pm}q_1^\pm+C^\pm a{\chi_{<R_0}}\gtrsim_{R_0} \mathbbm{1}_{V_{R_0}^\pm}.$$ Further, $q_1^\pm=\tilde{q}_1^\pm\circ \Phi^\pm,$ where $\tilde{q}_1^\pm\in C_c^\infty(T^*\R^3\setminus o)$. 
\end{Lemma} 

Next, we complete steps (1b) and (1c).
\begin{Lemma}[Non-trapped Escape Function Construction]\label{non-trap}
There exist $q_2^\pm\in C^\infty(T^*\R^3\setminus o)$ and $W^\pm\subset \Omega_\infty^\pm$ 
so that $V_{R_0}^\pm\cup W^\pm=T^*\R^3\setminus o,\ W^\pm\supset\{|x|>R_0\}$ and $$H_{p^\pm}q_2^\pm
\gtrsim c_j2^{-j}\mathbbm{1}_{W^\pm}, \qquad |x|\approx 2^{j}.$$ Further, $q_2^\pm= q_{in}^\pm+q_{out}^\pm,$ where $q_{in}^\pm=\tilde{q}_{in}^\pm\circ \Phi^\pm$, with $\tilde{q}_{in}^\pm\in C^\infty(T^*\R^3\setminus o)$ being supported in $\{|x|\leq 4R\}$ and $q_{out}^\pm\in S^0_{\operatorname{hom}}(T^*\R^3\setminus o)$.  
\end{Lemma}
We remark that the behavior of $(c_j)$ does not matter so much for $0\leq j<\log_2 R_0,$ as long as each corresponding $c_j$ is positive by compactness 

To complete the remaining steps and prove Lemma \ref{escape}, we proceed similarly to the work in \cites{Kof22}, with special attention paid to the new contributions of $s_{skew}.$ The presence of $\imag A$ in $s_{skew}$ did not occur in \cites{Kof22} and must be dealt with here. The symmetry conditions arise when one must balance the ability to leverage the sign on the damping for $|x|\leq 2R_0$ with the necessity to absorb the unsigned magnetic terms.
\begin{proof}[Proof of Lemma \ref{escape}]
First, we truncate the symbols to stay away from $\xi=0$:
 $$q^\pm_{j,>\lambda}=e^{-\sigma_j q_j^\pm}\chi_{|b^\pm|>\lambda},\ j=1,2,$$
  where $\sigma_1,\ \sigma_2\gg 1$. Unlike \cites{Kof22}, we need two parameters $\sigma_1$ and $\sigma_2$, as opposed to just one parameter; the additional parameter is needed to deal with unsigned first-order errors. 
The exponentiation is implemented for bootstrapping: Taking derivatives of the exponentials will provide multiplication by $\sigma_1$ and $ \sigma_2$. It is readily seen that $q^\pm_{j,>\lambda}\in S^0(T^*\R^3)$ via the chain rule.

We combine the symbols constructed on the individual light cones together as $$q(x,\tau,\xi)=(\tau-b^+)(q^-_{1,>\lambda}+q^-_{2,>\lambda})+(\tau-b^-)(q^+_{1,>\lambda}+q^+_{2,>\lambda}).$$ 
Calling  $$q_j=(\tau-b^+)q^-_{j,>\lambda}+(\tau-b^-)q^+_{j,>\lambda},$$ we can see that  
\begin{align*}
    &\left(H_p q+2\gamma\left(\tau \real a+\imag A_\alpha (g^{\alpha 0}\tau+g^{\alpha k}\xi_k)\right)q\right)\big|_{\tau=b^\pm}\\
    &=\left(H_p q_1+2\gamma\left(\tau \real a+\imag A_\alpha (g^{\alpha 0}\tau+g^{\alpha k}\xi_k)\right)q_1\right)\big|_{\tau=b^\pm}+\left(H_p q_2+2\gamma\left(\tau \real a+\imag A_\alpha (g^{\alpha 0}\tau+g^{\alpha k}\xi_k)\right)q_2\right)\big|_{\tau=b^\pm}\\
      &=H_p q_1 \big|_{\tau=b^\pm}\pm 2\gamma  (b^+-b^-)\left(b^\pm(\real a+\imag A_\alpha g^{\alpha 0})+\imag A_\alpha g^{\alpha k}\xi_k|_{\tau=b^{\pm}}\right)q_{1,>\lambda}^\pm\\
         &\qquad+H_p q_2 \big|_{\tau=b^\pm}\pm 2\gamma  (b^+-b^-)\left(b^\pm(\real a+\imag A_\alpha g^{\alpha 0})+\imag A_\alpha g^{\alpha k}\xi_k|_{\tau=b^{\pm}}\right)q_{2,>\lambda}^\pm.
\end{align*} We will work with each term in the last equality separately. First, we compute that
\begin{align*}
    H_p q_j\big|_{\tau=b^\pm}&=-(b^+-b^-)^2H_{p^\pm} q_{j,>\lambda}^\pm-(b^\pm-b^\mp)q_{j,>\lambda}^\pm(b_{\xi_j}^\pm b_{x_j}^\mp-b^\pm_{x_j}b^\mp_{\xi_j})\\
    &=\sigma_j (b^+-b^-)^2 q_{j,>\lambda}^\pm H_{p^\pm} q_j^\pm-(b^\pm-b^\mp)q_{j,>\lambda}^\pm(b_{\xi_j}^\pm b_{x_j}^\mp-b^\pm_{x_j}b^\mp_{\xi_j}).
\end{align*} By making $\sigma_1,\sigma_2$ sufficiently large, we may absorb error terms into the leading term with differing $j$ to obtain that
\begin{align*}
   H_p q_1\big|_{\tau=b^\pm}+H_p q_2\big|_{\tau=b^\pm}&\geq \frac{1}{2}\sigma_1(b^+-b^-)^2q_{1,>\lambda}^\pm H_{p^\pm} q_1^\pm+\frac{1}{2}\sigma_2(b^+-b^-)^2q_{2,>\lambda}^\pm H_{p^\pm} q_2^\pm,
\end{align*}
and so
\begin{align}
\label{light cone ineq}
    &\left(H_p q+2\gamma\left(\tau \real a+\imag A_\alpha (g^{\alpha 0}\tau+g^{\alpha k}\xi_k)\right)q\right)\big|_{\tau=b^\pm}\\
      &\gtrsim \sigma_1(b^+-b^-)^2q_{1,>\lambda}^\pm H_{p^\pm} q_1^\pm\pm 2\gamma  (b^+-b^-)\left(b^\pm(\real a+\imag A_\alpha g^{\alpha 0})+\imag A_\alpha g^{\alpha k}\xi_k|_{\tau=b^{\pm}}\right)q_{1,>\lambda}^\pm\nonumber\\
         &\qquad+\sigma_2(b^+-b^-)^2q_{2,>\lambda}^\pm H_{p^\pm} q_2^\pm\pm 2\gamma  (b^+-b^-)\left(b^\pm(\real a+\imag A_\alpha g^{\alpha 0})+\imag A_\alpha g^{\alpha k}\xi_k|_{\tau=b^{\pm}}\right)q_{2,>\lambda}^\pm.\nonumber
\end{align} 
We will consider each line in the above lower bound of (\ref{light cone ineq}) separately. First, observe that $$\frac{b^\pm}{b^\pm- b^\mp}\approx 1.$$

We first consider the terms in (\ref{light cone ineq}) featuring $q_{1, >\lambda}^\pm.$ Here, we want to leverage the damping via Lemma \ref{semibdd}. Splitting $\mathbb{R}^3$ into interior and exterior regions  permits one to use the weakly $\varepsilon$-damping dominant condition in the former and asymptotic flatness conditions in the latter. By choosing $\gamma$ sufficiently larger than $\varepsilon^{-1}\sigma_1$, we are able to apply Lemma \ref{semibdd}. Explicitly, for $|x|\approx 2^{\ell}$,
\begin{align}\label{highfreqp1 dd}
& \sigma_1(b^+-b^-)^2q_{1,>\lambda}^\pm H_{p^\pm} q_1\pm 2\gamma  (b^+-b^-)\left(b^\pm(\real a+\imag A_\alpha g^{\alpha 0})+\imag A_\alpha g^{\alpha k}\xi_k|_{\tau=b^{\pm}}\right)q_{1,>\lambda}^\pm\\
&=\sigma_1(b^+-b^-)^2q_{1,>\lambda}^\pm \left(H_{p^\pm} q_1^\pm+\left(\frac{2\gamma}{\sigma_1}\right)\left(\chi_{|x|<R_0}+\chi_{|x|>R_0}\right)\left(\frac{b^\pm}{b^\pm-b^\mp}\left(\real a+ g^{0\alpha}\imag A_\alpha\right)+\frac{\xi_k}{b^\pm-b^\mp}\imag A_\alpha g^{\alpha k}\right)\right)\nonumber\\
&\gtrsim\sigma_1(b^+-b^-)^2q_{1,>\lambda}^\pm \left(H_{p^\pm} q_1^\pm+\frac{2\gamma}{\sigma_1}(\varepsilon \real a\chi_{|x|<R_0}-c_\ell 2^{-\ell}\chi_{|x|>R_0})\right)\nonumber\\
&\gtrsim \sigma_1\chi_{|\xi|>\lambda}|\xi|^{2}\left(\mathbbm{1}_{{V_{R_0}^\pm}}-\frac{2\gamma}{\sigma_1} c_\ell2^{-\ell}\chi_{|x|>R_0}\right).
\nonumber
\end{align}

For the terms in (\ref{light cone ineq}) involving $q_{2,>\lambda}^\pm$, we may similarly split the physical space into an interior and exterior. In the former, the weakly $\varepsilon$-damping dominant applies once again (after which we merely use that the damping is non-negative for $|x|\leq 2R_0$ to drop the term altogether), and in the latter, $(c_\ell)$ provides quantitative control on the size of the perturbative terms. Using these strategies, applying Lemma \ref{non-trap} to both resulting pieces, and choosing $\sigma_2$ sufficiently larger than $\gamma$ gives that, for $|x|\approx 2^\ell$, 
\begin{align}\label{highfreqp2 dd}
    &\sigma_2(b^+-b^-)^2q_{2,>\lambda}^\pm H_{p^\pm} q_2\pm 2\gamma  (b^+-b^-)\left(b^\pm(\real a+\imag A_\alpha g^{\alpha 0})+\imag A_\alpha g^{\alpha k}\xi_k|_{\tau=b^{\pm}}\right)q_{2,>\lambda}^\pm\\
    &=\sigma_2\chi_{|x|<R_0}(b^+-b^-)^2q_{2,>\lambda}^\pm  \left(H_{p^\pm}  q_2^\pm+\left(\frac{2\gamma}{\sigma_2}\right)\frac{b^\pm}{b^\pm-b^\mp}\left(\real a+ g^{0\alpha}\imag A_\alpha\right)+\left(\frac{2\gamma}{\sigma_2}\right)\frac{\xi_k}{b^\pm-b^\mp}\imag A_\alpha g^{\alpha k}\right)\nonumber\\
    &\qquad+\sigma_2\chi_{|x|>R_0}(b^+-b^-)^2q_{2,>\lambda}^\pm  \left(H_{p^\pm} q_2^\pm+\left(\frac{2\gamma}{\sigma_2}\right)\frac{b^\pm}{b^\pm-b^\mp}\left(\real a+ g^{0\alpha}\imag A_\alpha\right)+\left(\frac{2\gamma}{\sigma_2}\right)\frac{\xi_k}{b^\pm-b^\mp}\imag A_\alpha g^{\alpha k}\right)\nonumber\\
    &\gtrsim \sigma_2\chi_{|x|<R_0}(b^+-b^-)^2q_{2, >\lambda}^\pm H_{p^\pm}q_2^\pm+\sigma_2\chi_{|x|>R_0}(b^+-b^-)^2q_{2, >\lambda}^\pm \left(H_{p^\pm}q_2^\pm-\frac{2\gamma}{\sigma_2}c_{\ell}2^{-\ell}\right)\nonumber\\
    &\gtrsim \sigma_2\chi_{|\xi|>\lambda} \chi_{|x|<R_0}c_\ell 2^{-\ell}|\xi|^{2}\mathbbm{1}_{W^\pm}+ \sigma_2\chi_{|\xi|>\lambda} \chi_{|x|>R_0}c_{\ell} 2^{-\ell}|\xi|^{2}\left(\mathbbm{1}_{W^\pm}-\frac{2\gamma}{\sigma_2}\right)\nonumber\\
    &\gtrsim \sigma_2\chi_{|\xi|>\lambda} \chi_{|x|<R_0}c_\ell 2^{-\ell}|\xi|^{2}\mathbbm{1}_{W^\pm}+ \sigma_2\chi_{|\xi|>\lambda} \chi_{|x|>R_0}c_{\ell} 2^{-\ell}|\xi|^{2}\mathbbm{1}_{W^\pm}\nonumber.
\end{align} 
Recall that $V_{R_0}^\pm\cup W^\pm=T^*\R^3\setminus o$ and $W^\pm\supset\{|x|>R_0\}$. We may readily combine (\ref{highfreqp1 dd}) and (\ref{highfreqp2 dd}) to directly get the estimate \begin{equation}\label{highfreqp2}
\left(H_pq+2\gamma \left(\tau \real a+\imag A_\alpha (g^{\alpha 0}\tau+g^{\alpha k}\xi_k)\right) q\right)\big|_{\tau=b^\pm}\gtrsim \mathbbm{1}_{|\xi|\geq\lambda}c_\ell 2^{-\ell}|\xi|^{2},\qquad |x|\approx 2^\ell
\end{equation} holds. All together, we require that $1\ll\sigma_1,\ \sigma_1\varepsilon^{-1}\ll\gamma\ll\sigma_2$.

From (\ref{highfreqp2}), we use
that $(c_\ell)$ is a slowly-varying, summable sequence  to conclude that 
$$\left(H_pq+2\gamma \left(\tau \real a+\imag A_\alpha (g^{\alpha 0}\tau+g^{\alpha k}\xi_k)\right) q\right)\big|_{\tau=b^\pm}\gtrsim \chi_{|\xi|>\lambda}\langle x\rangle^{-2}|\xi|^{2}.$$ 

The work on the elliptic set in e.g. \cites{MST20}, \cites{Kof22} applies without modification. Summarily, if we write 
\begin{equation}\label{ell set corr}H_pq+2\gamma \left(\tau \real a+\imag A_\alpha (g^{\alpha 0}\tau+g^{\alpha k}\xi_k)\right) q+pm=(a_0-m)\tau^2+\left(a_1+(b^++b^-)m\right)\tau+(a_2-b^+b^-m),
\end{equation}
where $a_j\in S^j(T^*\R^3)$, then we choose $${m}=-(b^+-b^-)^{-2}\left(a_1(b^++b^-)+2(a_0b^+b^-+a_2)\right).$$ Such an $m$ ensures that the quadratic polynomial (\ref{ell set corr}) in $\tau$ is concave up and has no real zeros. 
\end{proof}
The proof of Theorem \ref{high freq thm} is highly similar to the proof in \cites{Kof22} for the analogous result. There are a few additional terms to deal with, but the aforementioned work demonstrates how to deal with them, as they are lower-order. We will briefly summarize the argument in \cites{Kof22} and add in additional details for the new terms. First, we remark that one can readily reduce the theorem to a simplified estimate, which we state as a proposition.

\begin{prop}\label{high freq red} In order to prove Theorem \ref{high freq thm}, it suffices to prove that
\begin{equation}\label{reduced high freq}
\norm{v_{>\lambda}}_{LE^1_{<2R_0}}\lesssim C(\lambda,\gamma)\left(\norm{Pv}_{LE^*_c}^{1/2}\norm{v}_{LE^1}^{1/2}+\norm{v}_{L^2L^2}\right)+\gamma\lambda^{-1/2}\norm{v}_{LE^1}
\end{equation} for all $v$ supported in $\{|x|\leq 2R_0\},$ where $v_{>\lambda}=\chi_{|\xi|>\lambda}(D_x)v.$
\end{prop}
This reduction was shown in \cites{MST20} and \cites{Kof22}. We will not reproduce the proof here, but the idea is as follows:
\begin{enumerate}
    \item Use asymptotic flatness to reduce to the case of $Pu$ and $u[0]$ compactly-supported in $\{|x|\leq 2R_0\}$.
    \item Utilize a unit time interval and Duhamel's theorem to reduce to $u[0]=0$ and $Pu\in LE^*_c$.
    \item Remove the upper bound on the time integrals using a cutoff argument, making it sufficient to integrate in $t$ from $-\infty$ to $\infty.$
    \item Reduce to solutions supported in $\{|x|\leq 2R_0\}$ via standard exterior wave estimates.
    \item Add back in the low frequencies, then take $\lambda\gg \gamma$ and apply Young's inequality for products in order to bootstrap the $LE^1$ terms on the right into the left.
\end{enumerate}

Now, we will prove Theorem \ref{high freq thm}. During the proof, we will use $\inprod{\cdot,\cdot}$ to denote the $L^2_tL^2_x$ inner product. We also recall the $\gamma$-scaling that was introduced after Definition \ref{GCC}. The proof is a modified positive commutator argument. Analyzing $2\text{Im}\inprod {P v, \left(q^{\operatorname{w}}-\frac{i}{2}m^{\operatorname{w}}\right)v}$ will provide, up to lower-order error terms, precisely $\inprod{(H_{p}q-2is_{skew} q+mp)^{\operatorname{w}} v,v}$, which features the quantization of the upper bound term in the symbol estimate of Lemma \ref{escape}. Through a partition of unity argument in the frequency space, we may utilize the sharp G\r{a}rding inequality in view of Lemma \ref{escape} to establish a lower bound in the first-order local energy space. The original inner product involving $Pv$ may be dealt with using Cauchy-Schwarz and similar frequency analysis to partially compose the upper-bound of Theorem \ref{high freq thm}, while lower-order error terms provide lower-order obstructions that may either be absorbed on the lower-bound side via frequency localization arguments or are located on the upper-bound side (depending on the order of the error). 
\begin{proof}[Proof of Theorem \ref{high freq thm}] 
First, decompose $P$ its principal self-adjoint and skew-adjoint parts, plus an error, as $$P=\Box_g+P_{skew}+\tilde{P},$$ where $\tilde{P}$ represent all remaining lower-order terms. Integrating by parts gives us that
\begin{align}\label{initial comm}
 2\text{Im}
\inprod {P v, \left(q^{\operatorname{w}}-\frac{i}{2}m^{\operatorname{w}}\right)v} &=\inprod{i[\Box_g, q^{\operatorname{w}}]v,v}+\gamma\inprod{(P_{skew}q^{\operatorname{w}}+q^{\operatorname{w}}P_{skew})v,v}+\frac{1}{2}\inprod{(\Box_g m^{\operatorname{w}}+m^{\operatorname{w}}\Box_g)v,v}\\
&\qquad+\operatorname{lower-order\  terms}\nonumber.
\end{align}  Notice that the non-lower-order-terms on the right-hand side may be written, as a result of the Weyl calculus, as \begin{align*}\inprod{(H_{p}q-2is_{skew} q+mp)^{\operatorname{w}} v,v}+\inprod{{R}_0 v,v},\qwhere {R}_0\in\Psi^{0}(\R^4).
\end{align*} We will analyze each side separately. First, we record the full, explicit calculation for posterity, keeping the lower-order terms in (\ref{initial comm}) on the left-hand side of the equality: 
\begin{align}\label{comm}  2\text{Im}
&\inprod {P v, \left(q^{\operatorname{w}}-\frac{i}{2}m^{\operatorname{w}}\right)v} +\frac{i\gamma}{2}\inprod{[(\imag A_\alpha g^{\alpha\beta}D_\beta+D_\alpha g^{\alpha\beta}\imag A_\alpha), m^{\operatorname{w}}]v,v}+\frac{i\gamma }{2}\inprod{ [a D_t, m^{\operatorname{w}}]v,v}\\
&\qquad\qquad-i\gamma\inprod{[(\real A_\alpha g^{\alpha\beta}D_\beta+D_\alpha g^{\alpha\beta}\real A_\beta), q^{\operatorname{w}}]v,v}\nonumber\\
&\qquad\qquad-\frac{\gamma}{2}\inprod{\left(m^{\operatorname{w}}(\real A_\alpha g^{\alpha\beta}D_\beta+D_\alpha g^{\alpha\beta}\real A_\beta) +(\real A_\alpha g^{\alpha\beta}D_\beta+D_\alpha g^{\alpha\beta}\real A_\beta)m^{\operatorname{w}}\right)v,v}\nonumber\\
&\qquad\qquad+i\gamma^2\inprod{[\imag A_\alpha g^{\alpha\beta}\imag A_\beta-\real A_\alpha g^{\alpha\beta}\real A_\beta, q^{\operatorname{w}}]v,v}\nonumber\\
&\qquad\qquad-2\gamma^2\inprod{\left(q^{\operatorname{w}}(\imag A_\alpha g^{\alpha\beta}\real A_\beta)+(\imag A_\alpha g^{\alpha\beta}\real A_\beta)q^{\operatorname{w}}\right)v,v}\nonumber\\
&\qquad\qquad-\frac{\gamma^2}{2}\inprod{\left(m^{\operatorname{w}}(\real A_\alpha g^{\alpha\beta}\real A_\beta)+(\real A_\alpha g^{\alpha\beta}\real A_\beta)m^{\operatorname{w}}\right)v,v}\nonumber\\
&\qquad\qquad+\frac{\gamma^2}{2}\inprod{\left(m^{\operatorname{w}}(\imag A_\alpha g^{\alpha\beta}\imag A_\beta)+(\imag A_\alpha g^{\alpha\beta}\imag A_\beta)m^{\operatorname{w}}\right)v,v}\nonumber\\
&\qquad\qquad+\gamma^2\inprod{[\real A_\alpha g^{\alpha\beta}\imag A_\beta, m^{\operatorname{w}}]v,v}
+i\gamma^2\inprod{[ q^{\operatorname{w}},\real V]v,v}+\frac{i\gamma}{2}\inprod{[\imag V, m^{\operatorname{w}}]v,v}\nonumber\\ 
&\qquad\qquad-\frac{\gamma^2}{2}\inprod{(m^{\operatorname{w}}\real V+\real V m^{\operatorname{w}})v,v}-{\gamma^2}\inprod{(q^{\operatorname{w}}\imag V+\imag V q^{\operatorname{w}})v,v}\nonumber\\
&\qquad\qquad+\gamma\inprod{\imag a D_t  m^{\operatorname{w}} v,v}+\gamma\inprod{[\imag a, q^{\operatorname{w}}]D_t v,v}\nonumber\\
&=\inprod{i[\Box_{g},q^{\operatorname{w}}]v,v}+\frac{1}{2}\inprod{(\Box_{g}m^{\operatorname{w}}+m^{\operatorname{w}}\Box_{g})v,v}\nonumber \\ 
&\qquad\qquad+\gamma\inprod{\left(q^{\operatorname{w}}(\real a  D_t+\imag A_\alpha g^{\alpha\beta}D_\beta+D_\alpha g^{\alpha\beta}\imag A_\beta)+(\real a D_t+\imag A_\alpha g^{\alpha\beta}D_\beta+D_\alpha g^{\alpha\beta}\imag A_\beta) q^{\operatorname{w}}\right)v,v} \nonumber\\
&=\inprod{(H_{p}q-2is_{skew} q+mp)^{\operatorname{w}} v,v}+\inprod{{R}_0 v,v}\nonumber.
\end{align} Split $v$ as $$v=v_{>>\lambda}+v_{<<\lambda}:=\chi_{|\xi|+|\tau|>\lambda}(\partial)v+\chi_{|\xi|+|\tau|<\lambda} (\partial)v,$$ and 
 recall that $q,m$ are both supported at frequencies $|\xi|\geq\lambda.$ 
By plugging in the frequency decomposition, choosing $\gamma$ large enough, and applying the sharp G\r{a}rding inequality to $v_{>>\lambda}$ (justified by Lemma \ref{escape}), we obtain that 
$$\inprod{(H_{p}q-2is_{skew} q+mp)^{\operatorname{w}} v,v}\gtrsim \inprod{\left(\chi_{|\xi|>\lambda}\langle x\rangle^{-2}(|\xi|^2+\tau^2)\right)^{\operatorname{w}} v_{>>\lambda},v_{>>\lambda}}\\
-\norm{v_{>>\lambda}}_{H^{1/2}_{t,x}}^2+\inprod{Sv,v},$$ where $S\in\Psi^{-\infty}(\R^4)$ arises from the terms including $v_{<<\lambda}$ (since $\chi_{|\xi|+|\tau|<\lambda}\in S^{-\infty})$). Integrating by parts one time gives that 
\begin{align*}
\inprod{\left(\chi_{|\xi|>\lambda}\langle x\rangle^{-2}(|\xi|^2+\tau^2)\right)^{\operatorname{w}}v_{>>\lambda},v_{>>\lambda}}&\gtrsim\norm{\partial v_{>\lambda}}_{LE_{<2R_0}}^2+\inprod{{R}_1v_{>>\lambda},v_{>>\lambda}},\qwhere {R}_1\in \Psi^1(\R^4).
\end{align*} 
 All together, the right-hand side of (\ref{comm}) is bounded below by a multiple of
\begin{align*}
    \norm{\partial v_{>\lambda}}_{LE_{<2R_0}}^2-\left|\inprod{{R}_1v_{>>\lambda},v_{>>\lambda}}\right|-\norm{v_{>>\lambda}}_{H^{1/2}_{t,x}}^2-\left|\inprod{{R}_0v,v}\right|.
\end{align*}

Bounding the errors is performed using Plancherel's theorem. With more specificity, one may utilize the standard Sobolev mapping properties of pseudodifferential operators, the frequency localization, and the compact spatial support of $v$ to obtain that 
\begin{align*}
\left|\inprod{{R}_1v_{>>\lambda},v_{>>\lambda}}\right|+\norm{v_{>>\lambda}}_{H^{1/2}_{t,x}}^2\lesssim\lambda^{-1}\norm{v}_{LE^1}^2.
\end{align*}
One shows that $$\left|\inprod{{R}_0v,v}\right|\lesssim C(\lambda)\norm{v}_{L^2L^2}^2$$ in a similar manner, except that one cannot leverage frequency localization and may incur an implicit constant which depends on $\lambda$. Such a term will appear on the upper bound side of (\ref{reduced high freq}), so having such an implicit constant is permissible. 

Summarizing, we have shown that the right-hand side of (\ref{comm}) is bounded below by a multiple of 
\begin{align}\label{RHS comm bd}\norm{\partial v_{>\lambda}}_{LE_{<2R_0}}^2-C(\lambda)\norm{v}_{L^2L^2}^2-\lambda^{-1}\norm{v}_{LE^1}^2.
\end{align} Next, we consider the left-hand side of (\ref{comm}). By the Cauchy-Schwarz and Plancherel's theorem,
\begin{align*}\inprod {P v, \left(q^{\operatorname{w}}-\frac{i}{2}m^{\operatorname{w}}\right)v}&=\inprod {P v, \left(q^{\operatorname{w}}-\frac{i}{2}m^{\operatorname{w}}\right)v_{>>\lambda}}+\inprod{{S}v,v},\qquad\qquad {S}\in \Psi^{-\infty}(\R^4)\\
&\lesssim C(\lambda)\left(\norm{Pv}_{LE^*_c}\norm{v}_{LE^1}+\norm{v}_{L^2L^2}^2\right).
\end{align*} By once again applying frequency splitting, the remaining inner products on the left-hand side of (\ref{comm}) are of the form $$(\gamma+\gamma^2)\left(\inprod{\widetilde{R}_0v,v}+\inprod{\widetilde{R}_1v_{>>\lambda},v_{>>\lambda}}\right),\qquad \widetilde{R}_j\in \Psi^j(\R^4).$$ We have discussed how to bound both of these terms; namely, $$\left|\inprod{\widetilde{R}_0v,v}\right|+\left|\inprod{\widetilde{R}_1v_{>>\lambda},v_{>>\lambda}}\right|\lesssim C(\lambda)\norm{v}_{L^2L^2}^2+\lambda^{-1}\norm{v}_{LE^1}^2.$$ Factoring in the scalar coefficients of these inner products, we have demonstrated that the left-hand side of (\ref{comm}) is bounded above by a multiple of
\begin{align}\label{LHS comm bd}
C(\lambda,\gamma)\left(\norm{Pv}_{LE^*_c}\norm{v}_{LE^1}+\norm{v}_{L^2L^2}^2\right)+\gamma^2\lambda^{-1}\norm{v}_{LE^1}^2
\end{align}
Combining (\ref{RHS comm bd})-(\ref{LHS comm bd}) in application to (\ref{comm}) and completing the $LE^1$ norm on the lower-bound side provides (\ref{reduced high freq}). 
\end{proof}
\subsection{Remaining Frequency Analyses and Two-Point Local Energy Decay}\label{rem freq section}
In order to establish Theorem \ref{two point thm}, we require similar estimates in the low and medium frequency regimes. The damping does not play a meaningful role in either regime, as it may be treated as a lower-order perturbation term. Like in \cites{Kof22}, the relevant estimates from \cites{MST20} carry through. We will briefly summarize why this is the case, in lieu of full proofs. 

At low frequencies, the obstruction to local energy decay arises when $P$ has a \textit{resonance} at frequency zero (see Section \ref{resolvent section} for a precise definition and further discussion on spectral obstructions to local energy decay). A quantitative condition on the existence of corresponding zero resonant states is the following \textit{zero non-resonance condition}.
\begin{definition}
$P$ is said to satisfy a \textit{zero non-resonance condition} if there exists some $K_0$, independent of $t$, such that 
\begin{align}\label{zero res}
    \norm{u}_{\dot{H}^1}\leq K_0\norm{P_0 u}_{\dot{H}^{-1}}\qquad \forall u\in \dot{H}^1.
\end{align} 
\end{definition}
The elliptic operator $$P_0=P\big|_{D_t=0}=(D_j+A_j)g^{jk}(D_k+A_k)+A_0g^{0j}(D_j+A_j)+(D_j+A_j)g^{j0}A_0+(A_0)^2g^{00}+V$$ represents $P$ at time frequency zero, and 
we underscore that the damping does not appear. Hence, the damping has no bearing on whether or not the zero non-resonance condition holds. For example, if $P$ is stationary and asymptotically flat with $\imag A, \equiv 0$, $V>0,$ and $A_0=0$, then $P$ satisfies the zero non-resonance condition. This follows from Lemma 6.2iii in \cites{MST20}, which also features a more general condition. The relevant low frequency estimate is the following, and the corresponding theorem in \cites{MST20} is Theorem 6.1.
 \begin{Th}\label{low freq thm}
Let $P$ be an asymptotically flat damped wave operator which satisfies the zero non-resonance condition, and suppose that $\partial_t$ is uniformly time-like while the constant time slices are uniformly space-like. Then, the bound
\begin{equation}\label{low freq est}
\norm{u}_{LE^1}\lesssim\norm{\partial_t u}_{LE^1_c}+\norm{Pu}_{LE^*}
\end{equation} holds for all $u\in \mathcal{S}(\R^4)$.
\end{Th}
The proof of Theorem \ref{low freq thm} leverages weighted elliptic estimates for the flat Laplacian $\Delta$ in order to get similar estimates for $AF$ perturbations. Once again, the damping does not play a meaningful role. At frequency zero, it provides no contribution, and near frequency zero, it is absorbed by the error term in (\ref{low freq est}); the damping arises when estimating $P_0u$ by $Pu$ within a compact spatial set.

At medium frequencies, we require a weighted estimate which implies local energy decay for solutions supported at any range of time frequencies bounded away from both zero and infinity. This is rooted in Carleman estimates. The Carleman weights which we need are constructed in e.g. \cites{Bo18},  \cites{KT01}. The main medium frequency estimate is the following, and the corresponding theorem in \cites{MST20} is Theorem 5.4. We remark that the theorem does not imply an absence of embedded eigenvalues/resonances on the real line.
\begin{Th}\label{med freq thm}
Let $P$ be an asymptotically flat damped wave operator, and suppose that $\partial_t$ is uniformly time-like while the constant time slices are uniformly space-like. Then, for any $\delta>0$, there exists a bounded, non-decreasing radial weight $\varphi=\varphi(\ln (1+r))$ so that for all $u\in \mathcal{S}(\mathbb{R}^4)$, we have the bound 
\begin{multline}\label{med freq est}
\norm{(1+\varphi_+'')^{1/2}e^\varphi\left(\nabla u,\inprod{r}^{-1}(1+\varphi')u\right)}_{LE}+\norm{(1+\varphi')^{1/2}e^\varphi\partial_t u}_{LE}
   \\ \lesssim \norm{e^\varphi Pu}_{LE^*}
+\delta\left(\norm{(1+\varphi')^{1/2}e^\varphi u}_{LE}+\norm{\inprod{r}^{-1}(1+\varphi_+'')^{1/2}(1+\varphi')e^\varphi \partial_t u}_{LE}\right),
\end{multline} with the implicit constant independent of $\delta$.
\end{Th}
The proof of this theorem utilizes two intermediate Carleman estimates within two different regions of space, which may be combined using a cutoff argument in order to prove Theorem \ref{med freq thm}.
\begin{enumerate}
    \item \textbf{Within a large compact set:} The damping term is well-signed and readily absorbable as a perturbation due to the conditions on the weight $\varphi,$ which will be convex.
    \item \textbf{Outside of a large compact set:} Here, the damping is a small $AF$ perturbation, so the proof in \cites{MST20} follows through without any modification. Within this region, the authors of \cites{MST20} bend the weight to be constant near infinity in order to apply exterior wave estimates. This leads to breaking this case into three sub-regions: one where the Carleman weight is convex, a transition region where the conditions break in order to bend the weight to be constant near infinity, and a region near infinity where the weight is constant.
\end{enumerate}
The proofs of the Carleman estimates in the above regions are based on positive commutator arguments utilizing the self- and skew-adjoint parts of the conjugated operator $P_\varphi=e^\varphi Pe^{-\varphi}.$ 

The high, low, and medium frequency estimates are the key ingredients needed to establish Theorem \ref{two point thm}. As in \cites{MST20}, \cites{Kof22}, it sufficient to remove the Cauchy data at times $0$ and $T$ in order to prove Theorem \ref{two point thm}; we will elaborate on this momentarily. This makes it significantly easier to perform frequency localization. The pertinent result in \cites{MST20} is Theorem 7.1. 
\begin{Th}\label{uncond LED}
Let $P$ be a stationary, asymptotically flat damped wave operator which satisfies the zero non-resonance, geometric, and weakly $\varepsilon$-damping dominant conditions for some $\varepsilon>0$. Additionally, assume that $\partial_t$ uniformly time-like while the constant time slices are uniformly space-like. Then, the estimate
\begin{align}
    \label{uncond LED est}
    \norm{u}_{LE^1}\lesssim\norm{Pu}_{LE^*}
\end{align} holds for all $u\in\mathcal{S}(\R^4)$. 
\end{Th}
\begin{proof}
 We will utilize a time-frequency partition of unity. Let $0<\tau_0\ll 1$ and $\tau_1\gg 1$, which will be chosen with more precision shortly. Then, we can write \begin{align*}u&=\chi_{|\tau|< \tau_0}(D_t)u+\chi_{\tau_0< |\tau|< \tau_1}(D_t)u+\chi_{ |\tau|> \tau_1}(D_t)u=:Q_1u+Q_2u+Q_3 u.
 \end{align*}  Since $P$ is stationary, it commutes with each $Q_j,$ and so it suffices to show that 
   \begin{equation}\label{freq err}
   \norm{Q_ju}_{LE^1}\lesssim \norm{Pu}_{LE^*},\qquad j=1,2,3.
  \end{equation}
  
 First, we apply Theorem \ref{low freq thm} to $Q_1u$ and appeal to Plancherel's theorem in order to obtain that
 \begin{align*}
     \norm{Q_1u}_{LE^1}\lesssim \norm{\partial_t (Q_1u)}_{LE^1_c}+\norm{P(Q_1 u)}_{LE^*}\lesssim \tau_0\norm{Q_1u}_{LE^1_c}+\norm{P u}_{LE^*}.
 \end{align*}
 If $\tau_0$ is sufficiently small, then we may absorb the error term on the upper-hand side into lower-bound side, which provides (\ref{freq err}) for $j=1.$
 
  We proceed similarly with $Q_2u$ via Theorem \ref{med freq thm}:
 \begin{align*}
    &\norm{(1+\varphi_+'')^{1/2}e^\varphi\left(\nabla Q_2u,\inprod{r}^{-1}(1+\varphi')Q_2u\right)}_{LE}+\norm{(1+\varphi')^{1/2}e^\varphi\partial_t Q_2u}_{LE}
   \\ 
   &\lesssim \norm{e^\varphi P(Q_2u)}_{LE^*}
    +\delta\left(\norm{(1+\varphi')^{1/2}e^\varphi Q_2u}_{LE}+\norm{\inprod{r}^{-1}(1+\varphi_+'')^{1/2}(1+\varphi')e^\varphi \partial_t Q_2u}_{LE}\right)\\
    &\lesssim  \norm{e^\varphi Pu}_{LE^*}+\frac{\delta}{\tau_0}\norm{(1+\varphi')^{1/2}e^\varphi \partial_t Q_2u}_{LE}+\delta\tau_1 \norm{\inprod{r}^{-1}(1+\varphi_+'')^{1/2}(1+\varphi')e^\varphi Q_2u}_{LE}.
\end{align*}
By choosing $\delta$ sufficiently small, last two terms absorb into the left-hand side. Since $\varphi$ is bounded and $\varphi'\geq 0$, we obtain (\ref{freq err}) for $j=2.$
 
 Finally, we apply Theorem \ref{high freq thm} to $Q_3u(t-T/2)$:
 $$\norm{Q_3u}_{LE^1[-T/2,T/2]}\lesssim\norm{\partial (Q_3u)(-T/2)}_{L^2}+\norm{\inprod{x}^{-2}u}_{LE[-T/2,T/2]}+\norm{P(Q_3u)}_{LE^*[-T/2,T/2]}.$$ 
 Taking the limit as $T\rightarrow\infty$ and then applying Plancherel's theorem give that 
 $$\norm{Q_3u}_{LE^1}\lesssim \tau_1^{-1}\norm{Q_3u}_{LE^1}+\norm{Pu}_{LE^*}.$$ If $\tau_1$ is large enough, then the error term on the right absorbs into the left, giving (\ref{freq err}) for $j=3.$
\end{proof}
We underscore how important it is that $\delta$ may be chosen arbitrarily in Theorem \ref{med freq thm}: It allowed for compatibility with the high and low frequency estimates regardless of how high or low the frequency thresholds became, respectively (so long as they were away from zero and infinity).

 As in Section 7 of \cites{MST20}, one proves that Theorem \ref{uncond LED} implies Theorem \ref{two point thm} by fixing $u$ and constructing a function $v$ which matches the Cauchy data of $u$ at times $0$ and $T$ (and satisfies an appropriate bound) which allows one to apply (\ref{uncond LED est}) to $u-v.$ This construction is performed using a partition of unity on the support of $u[0]$, $u[T]$, and $Pu.$ In particular, one splits into an interior region $\{|x|<2R_0\}$ and an exterior region $\{|x|>R_0\}$. The damping is non-problematic in the interior and is small in the exterior (as with the other lower-order terms). The latter fact is important since a time reversal argument is used in the exterior \textit{alone}, avoids turning the damping into a driving force.
 
 In the interior, one utilizes a unit time interval partition of unity $\{\chi_j\}$ and analyzes the equations  \newline\noindent $Pv_j=\chi_j f,$  matching the data at times $0$ and $T$ with the first and last of elements in the partition, respectively. One then generates the desired approximate of $u$ via $\sum \chi_jv_j.$ In the exterior region, one chooses an appropriate small $AF$ perturbation of $\Box$ which matches $P$ in the exterior. If one considers the same differential equation (same data and forcing) but replaces $P$ by the perturbation, one obtains good bounds via local energy decay. By truncating the solution appropriately to $|x|>R_0$ and $t<T,$ one obtains the desired approximate in the exterior. 
  \subsection{An Energy Dichotomy}\label{dichot section}
Here, we apply Theorem \ref{two point thm} in order to prove the energy dichotomy present in Theorem \ref{energy dich}. Very little deviation is needed from the strategy given in \cites{MST20}, although we provide a proof to make this self-contained. One deviation is that we must take advantage of the damping condition. In particular, we must use the damping to absorb time derivative error terms for $|x|\leq 2R_0$, outside of which we may use asymptotic flatness (for space derivative terms, one can use Young's inequality for products, at the expense of shrinking $\varepsilon$). 
 \begin{proof}    
  As mentioned in Section \ref{ee defs}, the energy associated to $P_0$ is not coercive, so we will instead symmetrize. To that end, split $P$ and $P_0$ into self- and skew-adjoint parts $$P=P^s+P^a,\qquad P_0=P_0^{s}+P_0^a,$$ respectively, and define the energy of the symmetric part of $P$ as 
  $${E}^s[\boldsymbol{u}](t)=\int\limits_{\R^3}P_0^s u\bar{u}-g^{00}|\partial_t u|^2\, dx.$$ The explicit expressions for these operator splittings are not so important, but their symmetric/skew-symmetry properties will be convenient. 
  
  Due to the symmetry of $P_0^s$, integration by parts yields that \begin{equation}\label{symm energy exp}E^s[\boldsymbol{u}](t)=E^s[\boldsymbol{u}](0)+2\real \int\limits_0^t\int\limits_{\R^3} \partial_s u\overline{P^su}\, dxds.\end{equation} We must analyze the last term. As has been the theme, one may split space into an interior region (where damping is positively signed and the damping-dominance holds) and an exterior region (where we may apply asymptotic flatness). It follows from the operator splitting; H\"older's inequality; Young's inequality for products (with constant $\varepsilon^k$ for any $k\in (0,1)$); the assumed $\varepsilon$-damping dominance with respect to $A,\ \nabla A,$ and $V$; and asymptotic flatness that
  \begin{align} \label{alm cons ee}
   \int\limits_0^t\int\limits_{\R^3} \partial_s u\overline{P^su}\, dxds&=\int\limits_0^t\int\limits_{\R^3} \partial_s u\overline{f}\, dxds-   \int\limits_0^t\int\limits_{\R^3} \partial_s u\overline{P^au}\, dxds\\
   &=\int\limits_0^t\int\limits_{\R^3} \partial_s u\overline{f}\, dxds  -C(\varepsilon)\int\limits_0^t\int\limits_{B_{2R_0}(0)} a|\partial_s u|^2\, dxds+D(\varepsilon, \textbf{c})\norm{u}_{LE^1}^2. \nonumber
   \end{align}
   In the above, $C(\varepsilon)= 1-c\varepsilon^{1-k}>0$ for fixed $c$ (shrinking $\varepsilon$ if necessary guarantees the positivity) and
  $D(\varepsilon, \textbf{c})=\mathcal{O}(\max(\varepsilon^{k}, \textbf{c}))$.

   Pairing this with (\ref{almost coerc}) and dropping the damping term gives that 
  \begin{align*}
      \norm{\partial u}_{L^\infty L^2[0,T]}^2
      \lesssim\norm{\partial u(0)}_{L^2}^2+\norm{\partial^{\leq 1}u(T)}_{L^2_c}^2+D(\varepsilon, \textbf{c})\norm{u}_{LE^1[0,T]}^2+\int\limits_0^T\int\limits_{\R^3} |\partial_t u||f|\, dxdt. 
  \end{align*} 
   The Schwarz inequality and H\"older's inequality imply that 
    \begin{align}\label{energy dich starting est}
      \norm{\partial u}_{L^\infty L^2[0,T]}
      \lesssim\norm{\partial u(0)}_{L^2}+\norm{\partial^{\leq 1}u(T)}_{L^2_c} +(\delta+D(\varepsilon, \textbf{c}))\norm{u}_{LE^1[0,T]}+\delta^{-1}\norm{f}_{LE^*+L^1L^2[0,T]},
  \end{align} 
  where $\delta>0$ is arbitrary.
   Now, we will make use of the two-point local energy estimate (\ref{two point LED}), valid as the conditions to apply Theorem \ref{two point thm} hold by assumption.
  By combining (\ref{energy dich starting est}) with (\ref{two point LED}) and choosing $\delta$ sufficiently small, we obtain that $$\norm{\partial u(T)}_{L^2}^2\lesssim\norm{\partial u(0)}_{L^2}^2+\norm{\partial^{\leq 1}u(T)}_{L^2_c}^2+\norm{f}^2_{LE^*+L^1L^2[0,\infty)}.$$ On the other hand, (\ref{two point LED}) directly gives that $$ \norm{\partial^{\leq 1}u(t)}_{L^2L^2_c[0, T]}^2\lesssim\norm{u}_{LE^1_c[0,T]}^2\lesssim \norm{\partial u(0)}_{L^2}^2+\norm{\partial u(T)}_{L^2}^2+\norm{f}^2_{LE^*+L^1L^2[0,\infty)},$$ and so  \begin{align}\label{dich eqn 1}
     \norm{\partial u}_{L^2L^2[0,T]}^2 \lesssim \norm{\partial u(T)}_{L^2}^2+(T+1)\left(\norm{\partial u(0)}_{L^2}^2+\norm{f}^2_{LE^*+L^1L^2[0,\infty)}\right).\end{align}
  Notice that if we call  $E(T)=\norm{\partial u}_{L^2L^2[0,T]}^2,$  then (\ref{dich eqn 1}) gives that \begin{align}\label{dich eqn 2}E'(T)=\norm{\partial u(T)}_{L^2}^2\geq \alpha E(T)-(T+1)\left(\norm{\partial u(0)}_{L^2}^2+\norm{f}^2_{LE^*+L^1L^2[0,\infty)}\right),\end{align} where $\alpha>0$ is a constant. 
  
  We will consider two cases. 
  First, we assume that  $$E(T)<2\alpha^{-1}(T+1)\left(\norm{\partial u(0)}_{L^2}^2+\norm{f}^2_{LE^*+L^1L^2[0,\infty)}\right)$$ for all $T>0$. In such a case, we have that, in particular, $$T^{-1}
  \norm{\partial u}_{L^2L^2[0,T]}^2\lesssim \norm{\partial u(0)}_{L^2}^2+\norm{f}^2_{LE^*+L^1L^2[0,\infty)}.$$ By the mean-value theorem for integrals, there exists a sequence $(T_j)$ so that $T_j\rightarrow \infty$ as $j\rightarrow\infty$, and $$\norm{\partial u(T_j)}_{L^2}^2\lesssim \norm{\partial u(0)}_{L^2}^2+\norm{f}^2_{LE^*+L^1L^2[0,\infty)}.$$ We may use (\ref{two point LED}) and let $T_j\rightarrow\infty$ to conclude that local energy decay holds. 
 
 Next, we consider if \begin{align}\label{dich eqn 3}
     E(T')\geq2\alpha^{-1}(T'+1)\left(\norm{\partial u(0)}_{L^2}^2+\norm{f}^2_{LE^*+L^1L^2[0,\infty)}\right)
 \end{align} for some $T'>0$. Then, $E(T)$ must bound (\ref{dich eqn 3}) from above for all $T\geq T'$, since (\ref{dich eqn 2}) implies that $E(T)$ is increasing for all $T>T'.$ 
By applying an integrating factor argument to the differential inequality (\ref{dich eqn 2}) and using (\ref{dich eqn 3}), we get that 
 $$E(T)\gtrsim_{T'} e^{\alpha T}\left(\norm{\partial u(0)}_{L^2}^2
 +\norm{f}^2_{LE^*+L^1L^2[0,\infty)}\right),\qquad T\geq T',\qquad T'\gg 1.$$ Combining this with (\ref{dich eqn 1}) gives the exponential growth for large enough $T$.
 \end{proof}

\section{Resolvent Theory and Local Energy Decay}\label{resolvent section}

 In this section, we will introduce the spectral theory required to prove Theorem \ref{LED thm}. This follows from the resolvent formalism introduced in \cites{MST20} and the corresponding scattering theory. We will summarize the relevant parts of their work here and provide details that were either omitted or are instructive to repeat. In particular, the results necessary to prove Theorem \ref{LED thm} are based on the same frequency estimates present in both \cites{MST20} and here, so our work requires little deviation. Throughout this section, we assume that $P$ is stationary. 

Consider $Pu=0.$ One arrives at the \textit{stationary problem} by studying ``mode solutions'' of the form $u(t,x)=u_\omega(x)e^{i\omega t}$, where $\omega\in\C$ (equivalently, one replaces $D_t$ by $\omega$). Plugging such $u$ into the given homogeneous equation generates the \textit{stationary }equation $$P_\omega u_\omega=0,\qwhere P_\omega=\Delta_{g,A}+W(x,D_x)+\omega B(x,D_x)+g^{00}\omega^2,$$
and \begin{align*}
    \Delta_{g,A}&=(D_j+A_j)g^{jk}(D_k+A_k),\\
    W(x,D_x)&=A_0g^{0j}(D_j+A_j)+(D_j+A_j)g^{j0}A_0+(A_0)^2g^{00}+V,\\
        B(x,D_x)&=g^{0j}(D_j+A_j)+(D_j+A_j)g^{j0}+2A_0g^{00}+ia.
\end{align*}
The \textit{resolvent} $R_\omega$ is defined as the inverse of $P_\omega$ when such an inverse exists. More explicitly, if we consider the homogeneous Cauchy problem $$Pu=0,\qquad u(0)=0,\qquad -g^{00} \partial_tu(0)=f,$$ then we may formally define $R_\omega$ via the Fourier-Laplace transform of $u$, i.e. $$R_\omega f=\int\limits_0^\infty e^{-i\omega t} u(t)\, dt=:\mathcal{F}_{t\rightarrow\omega} (\mathbbm{1}_{[0,\infty)}(t) u), \qquad\omega\in\C.$$ One may check via formal integration by parts that both definitions of $R_\omega$ are consistent. In this section, we will take $f$ to be in either $L^2$ or $\lecal^*$, and it will be clear from context which is the case. It remains to make rigorous sense of $R_\omega$ as a well-defined bounded operator.

From the global energy bounds and Gronwall's inequality, it follows that $u$ satisfies the crude estimate \begin{equation}\label{a priori}\norm{\partial u(t)}_{L^2}\lesssim e^{c t} \norm{f}_{L^2},\qquad c\geq 0.
\end{equation}
Using (\ref{a priori}) and the Minkowski integral inequality, we obtain that 
\begin{align}\label{init resolvent est pf 1}
\norm{R_\omega f}_{\dot{H}^1}\leq \int\limits_0^\infty e^{\imag  \omega t}\norm{\nabla u(t,\cdot)}_{L^2}\, dt\lesssim \int\limits_0^\infty e^{\imag  \omega t}e^{c t}\norm{f}_{L^2}\, dt \lesssim|\imag  \omega+c|^{-1}\norm{f}_{L^2},\qquad \imag\omega+c<0.
\end{align} Meanwhile, integrating by parts once provides that  $\omega R_\omega f=-i\mathcal{F}_{t\rightarrow\omega}(\mathbbm{1}_{[0,\infty)}\partial_t u).$ Taking the $L^2$ norm of the above and performing the same work as in (\ref{init resolvent est pf 1}) yields an identical upper bound. Combining these estimates together gives the inequality
\begin{align}\label{unif res strip}
     \norm{R_\omega f}_{\dot{H}_\omega^1}\lesssim|\imag  \omega+c|^{-1}\norm{f}_{L^2},\qquad   \imag\omega+c<0.
\end{align}

Hence, we may validly define the resolvent as a bounded operator from $L^2$ to  $\dot{H}_\omega^1$, provided that $\omega$ is in the range given in (\ref{unif res strip}). Notice that if the uniform energy bound (\ref{unif energy}) holds, then the resolvent is holomorphic in the lower half-plane $\mathcal{H}:=\{\omega\in\C: \imag\omega<0\}$ and satisfies the bound 
\begin{align}\label{unif energy res bd}\norm{R_\omega}_{L^2\rightarrow\dot{H}^1_\omega}\lesssim |\imag\omega|^{-1},\qquad\omega\in\mathcal{H}.\end{align} If the uniform energy bound does not hold, then one is only guaranteed meromorphic continuation to $\mathcal{H}.$ Due to this tie-in with uniform energy bounds, we will refer to (\ref{unif energy res bd}) as the \textit{uniform energy resolvent bound}. There is also an analogous resolvent bound to local energy decay, which we state as a theorem and will not prove here; see the proof of Theorem 2.3 in \cites{MST20} for more. It is largely a consequence of Plancherel's theorem, along with utilization of facts which we will discuss after the statement of the theorem. The damping plays no meaningful role here.
\begin{Th}\label{energy res equiv}
Local energy decay holds for a stationary damped wave operator $P$ if and only if $R_\omega$ satisfies the \textit{local energy resolvent bound}
\begin{align}\label{le res bd}
\norm{R_\omega}_{\lecal^*\rightarrow\lecal^1_{\omega}}\lesssim 1,\qquad\omega\in\mathcal{H}.
\end{align}
\end{Th}
 We observe that if the uniform energy resolvent bound holds, then the local energy resolvent bound holds for $\imag\omega\lesssim -1$, since  \begin{align}\label{unif to le}
     \norm{R_\omega f}_{\mathcal{LE}_\omega^1}&\lesssim
    \norm{ R_\omega f}_{\dot{H}^1}+|\omega|\norm{R_\omega f}_{L^2}=\norm{R_\omega f}_{\dot{H}_\omega^1}\lesssim |\imag\omega|^{-1}\norm{f}_{L^2}\lesssim \norm{f}_{\mathcal{LE}^*}.
 \end{align} On the other hand (\ref{le res bd}) implies (\ref{unif energy res bd}); Fredholm theory implies that $R_\omega$ is bounded and holomorphic from $L^2$ to $\dot{H}^1_\omega$, and the precise operator norm bound can by obtained by splitting $f\in L^2$ as $$f=\chi_{<|\imag\omega|^{-1}}f+\chi_{>|\imag\omega|^{-1}}f,$$ then using (\ref{le res bd}) and (\ref{unif res strip}). Hence, the uniform energy and local energy resolvent bounds have close relation. 
 We previously alluded to two spectral obstructions to local energy decay (and hence local energy resolvent bounds), which we outline more precisely now.
 
Elements in the kernel of $P_\omega$ living in $L^2$ correspond to finite rank poles $\omega$ of $R_\omega$. Given that $P_\omega$ is an elliptic operator, such eigenfunctions live in $H^s$ for all $s\in\R$, yet the corresponding mode solution to $Pu=0$ must possess exponential growth in time since $\omega\in\mathcal{H}$. In particular, the initial energy is finite, yet both terms on the lower-bound side of the local energy decay estimate (\ref{LED}) are unbounded as $T\rightarrow\infty$. This behavior violates both uniform energy bounds and local energy decay. Since the eigenvalues have negative imaginary part, we will call these \textit{negative eigenfunctions}.
\begin{definition} 
A negative eigenfunction for $P$ is a non-zero $u_\omega\in L^2$ such that $P_\omega u_\omega=0$, with $\omega\in \mathcal{H}.$
\end{definition} 
The corresponding eigenvalues are isolated and contained within a relatively compact set in $\mathcal{H}$, which follows due to (\ref{unif res strip}) and the high frequency estimate (\ref{high freq est}).

Next, we notice that the local energy resolvent bound must hold uniformly up to the real line in order to obtain local energy decay, per Theorem \ref{energy res equiv}. To that end, we have another potential obstruction - a failure to obtain the continuous extension of $R_\omega$ to $\R$. First, we consider when one takes the limit as the spectral parameter approaches a non-zero real value.
\begin{definition}
An embedded resonant state for $P$ is a non-zero $u_\omega\in\lecal^1_\omega$, with $\omega\in\R\setminus\{0\}$ being the corresponding embedded resonance, satisfying the outgoing/Sommerfeld radiation condition
\begin{equation}\label{sommer}\norm{\inprod{x}^{-1/2}(\partial_r+i\omega)u_\omega}_{L^2(A_j)}\rightarrow 0\qquad\operatorname{ as }j\rightarrow\infty
\end{equation} such that $$P_\omega u_\omega=0.$$ 
\end{definition}
The outgoing radiation condition as stated here is a variant of the standard one, adapted to the dyadic structure of our spaces. It acts as a boundary condition at infinity to ensure unique solutions to the above problem for a fixed $\omega$.  Such states are subtle obstructions, since the corresponding mode solution does not generally have finite energy (as $u_\omega$ is only guaranteed to live in $\lecal^1_\omega$).  
However, as explained in \cites{MST20}, one may perform a truncation procedure and utilize normalized/Regge-Wheeler-type coordinates (see \cites{Tat13}) to produce functions $\chi_{>1}(t-r)e^{i\omega t}u_\omega$ whose energy exhibits a growth rate of $t^{1/2}$ yet whose image under $P$ lives in $L^1L^2.$  
Hence, they represent an explicit obstruction to uniform and local energy bounds. The non-existence of embedded resonant states is equivalent to (\ref{le res bd}) near the punctured real line and to a limiting absorption principle. In our context, the corresponding result is as follows.

\begin{Th}\label{res equiv}
Let $P$ be a stationary, asymptotically flat damped wave operator which satisfies the geometric control condition and is weakly $\varepsilon$-damping dominant for some $\varepsilon>0$. Additionally, assume that $\partial_t$ is uniformly time-like while the constant time slices are uniformly space-like.
For any $\omega_0\in\mathbb{R}\setminus\{0\}$, the following are equivalent:
\begin{enumerate}
    \item $\omega_0$ is not a resonance.
    \item The bound $$\norm{u}_{\lecal^1_{\omega_0}}\lesssim\norm{P_{\omega_0}u}_{\lecal^*}$$ holds for all $u\in\lecal^1_{\omega_0}$ satisfying the outgoing radiation condition (\ref{sommer}).
    \item The local energy resolvent bound (\ref{le res bd}) holds uniformly for $\omega\in\mathcal{H}$ near $\omega_0$, and the limit $$R_{\omega_0} f=\lim_{\mathcal{H}\ni\omega\rightarrow\omega_0}R_\omega f,\qquad f\in\lecal^*$$ converges strongly on compact sets and satisfies the outgoing radiation condition (\ref{sommer}).
\end{enumerate}
\end{Th}
We sketch the proof, as it is similar to that given in \cites{MST20} for their version of the result (which is Proposition 2.5 in their work); only minor alterations are necessary to account for the damping and the lack of the non-trapping hypothesis. The result is perturbative of the work in \cites{Kof22}, just as the version in \cites{MST20} was perturbative of the case where $P=\Box_g$. 
\begin{proof}
First, consider when $A,V\equiv 0$ and $a\equiv 0$ for $|x|>2R_0.$ In this scenario, we are in the setting of \cites{Kof22} and obtain full local energy decay. By Theorem \ref{energy res equiv}, we have the local energy resolvent bound (\ref{le res bd}), in which case all three statements in Theorem \ref{res equiv} hold. 

In general, let $\widetilde{P}$ denote the principal part of $P$, and consider $\widetilde{Q}=\widetilde{P}+\chi_{|x|<R_0}iaD_t$. Note that our starting case applies to $\widetilde{Q}$, hence it satisfies local energy decay. Let $\widetilde{R}_\omega$ denote the resolvent of $\widetilde{Q}_\omega$, which is holomorphic in $\mathcal{H}$ and continuous up to the real line. We seek a solution to $P_\omega u=f$ of the form $u=\widetilde{R}_\omega g$ for some $g$ (that is, $u$ is in the range of the resolvent and hence outgoing). If $Q_\omega=P_\omega-\widetilde{Q}_\omega$, then $$P_\omega u=f\qquad\text{ if and only if }\qquad (I+Q_\omega \widetilde{R}_\omega)g=f.$$
The family of operators $Q_\omega\widetilde{R}_\omega$ is compact from $\lecal^*$ to $\lecal^*$, 
holomorphic in the lower-half plane, and continuous up to the real line. 
 Further, $I+Q_\omega \widetilde{R}_\omega$ is invertible for $-\imag\omega\gg 1$ by a Neumann series argument; to see this, note that $Q_\omega$ is first-order, then use a similar estimate to (\ref{unif to le}). Thus, $I+Q_\omega \widetilde{R}_\omega$ is a family of zero-index Fredholm operators. The conclusion of the theorem follows from the analytic Fredholm theorem and its consequences.
\end{proof}
Next, we consider when the spectral parameter approaches 0. In this case, we must replace the outgoing radiation condition (which is not meaningful when $\omega=0$) with a new condition which still limits the asymptotics of the functions.
\begin{definition}
A \textit{zero resonant state} for $P$ is a non-zero $u\in \lecal_0$ such that $P_0u=0$. If, in addition, $u\in L^2,$ then we call $u$ a \textit{zero eigenfunction}.
\end{definition}
Such resonant states are annihilated by $P$ while having finite energy. However, they also possess an unbounded $LE^1$ norm as $T\rightarrow\infty$, which violates local energy decay, as well as two-point local energy decay. As shown in \cites{MST20} (see Proposition 2.10 in their work), one also has an analogous characterization of a resonance at zero to that of the non-zero real resonances in Theorem \ref{res equiv}. We record their result without proof, for reasons which immediately follow the statement of the theorem.
\begin{Th}\label{zero res equiv}
Let $P$ be an asymptotically flat damped wave operator satisfying the geometric control condition and $\partial_t$ be uniformly time-like. Then, the following are equivalent:
\begin{enumerate}
    \item $0$ is not a resonance.
    \item The zero resolvent bound  \indent\begin{align}\label{weighted low freq}\norm{u}_{\lecal^1}\lesssim\norm{P_{0}u}_{\lecal^*}\end{align} holds for all $u\in\lecal^1_{0}$. 
    \item The zero non-resonance condition (\ref{zero res}) holds.
    \item The local energy resolvent bound (\ref{le res bd}) holds uniformly for $\omega\in\mathcal{H}$ near $0$, and the limit $$R_{0} f=\lim_{\mathcal{H}\ni\omega\rightarrow 0}R_\omega f,\qquad f\in\lecal^*$$ converges strongly on compact sets. 
    \item The stationary local energy decay estimate $$\norm{u}_{LE^1[0,T]}+\norm{\partial u}_{L^\infty L^2[0,T]}\lesssim\norm{\partial u(0)}_{L^2}+\norm{\partial u(T)}_{L^2}+\norm{\partial_t u}_{LE[0,T]}+\norm{Pu}_{LE^*+L^1L^2[0,T]},$$ with the implicit constant being independent of $T$.
\end{enumerate}
\end{Th} The damping does not arise in $P_0$ and simply acts as an absorbable lower-order term when transitioning back to $P$, so the added context of this paper has absolutely no effect on the work present in \cites{MST20}. In particular, we recover all of the same low frequency work, which is what is utilized for this theorem. Again, we refer the reader to the aforementioned work for more.

Theorem \ref{low freq thm} gives the implication (3)$\implies(5)$, whereas repeating the proof of Theorem \ref{low freq thm} on the resolvent side produces (3)$\implies(4).$ Next, we remark that the bound present in (\ref{weighted low freq}) is more \textit{directly} related to zero not constituting a resonance than the zero non-resonance condition (\ref{zero res}). However, their equivalence can be seen from the fact that $P_0$ is an elliptic $AF$ perturbation of the Laplacian.

Now that we have outlined the obstructions to local energy decay, we proceed with a proof of Theorem \ref{LED thm}. By Theorem \ref{energy res equiv}, it suffices to establish (\ref{le res bd}). Once again, we may follow \cites{MST20}, as the proof only requires the resolvent theory and frequency estimates.
\begin{proof}[Proof of Theorem \ref{LED thm}]
Since there are no negative eigenfunctions, the resolvent is bounded and holomorphic in $\mathcal{H}$, which implies that $$\norm{R_\omega f}_{\dot{H}^1_\omega}\leq C(\imag\omega)\norm{f}_{\lecal^*},\qquad\omega\in\mathcal{H},$$ for some constant $C(\imag\omega)$ depending on $\imag\omega$ which is bounded away from the real line. This is not a uniform bound and a priori may become unbounded as $\imag\omega\rightarrow 0$, but it \textit{does} holds up to any neighborhood of $\R$. Similar to the work in (\ref{unif to le}), this implies the local energy resolvent bound (\ref{le res bd}) away from any neighborhood of the real line. It remains to establish (\ref{le res bd}) within a strip in the lower half-plane sufficiently close to $\R$. We break this strip into three exhaustive cases.
\vskip .05in
\noindent\textbf{Case I: $|\omega| \gg 1$}
\vskip .05in\noindent This is in the high frequency regime, motivating us to use the high frequency estimate (\ref{high freq est}), which is applicable by hypothesis. Let $u$ solve $P_\omega u=f,$ and call $v=e^{i\omega t}u.$ Then, $v$ solves $Pv=g$, where $g=e^{i\omega t}f.$ We will apply the high frequency estimate to the interval $[-T,0]$. More precisely, if we call $\tilde{v}(t,x)=v(t-T,x)$, then this solves $P\tilde{v}=\tilde{g},$ where $\tilde{g}(t,x)=g(t-T,x).$ Applying Theorem \ref{high freq thm} to $\tilde{v}$ provides that
 $$ \norm{\tilde{v}}_{LE^1[0,T]}+\norm{\partial \tilde{v}}_{L^\infty L^2 [0,T]}\lesssim \norm{\partial \tilde{v}(0)}_{L^2}+\norm{\inprod{x}^{-2} \tilde{v}}_{LE[0,T]}+\norm{\tilde{g}}_{LE^*+L^1L^2 [0,T]}.$$
 We immediately calculate that
 \begin{align*}
     \norm{\tilde{v}}_{LE^1[0,T]}&=\left(\frac{e^{2T\imag\omega}-1}{2\imag\omega}\right)^{1/2}\norm{u}_{\mathcal{LE}^1_{\omega}}\\
     \norm{\partial \tilde{v}}_{L^\infty L^2 [0,T]}&= \norm{u}_{\dot{H}^1_{\omega}}\\
     \norm{\partial \tilde{v}(0)}_{L^2}&= e^{T\imag\omega}\norm{u}_{\dot{H}^1_\omega}\\
     \norm{\inprod{x}^{-2} \tilde{v}}_{LE[0,T]}&=\left(\frac{e^{2T\imag\omega}-1}{2\imag\omega}\right)^{1/2}\norm{\inprod{x}^{-2}u}_{\mathcal{LE}}\\
     \norm{\tilde{g}}_{LE^*+L^1L^2 [0,T]}&=\norm{f}_{\left(\frac{\exp (2T\imag\omega)-1}{2\imag\omega}\right)^{-1/2}\mathcal{LE^*}+\left(\frac{\exp (T\imag\omega)-1}{\imag\omega}\right)^{-1}{L^2}}.
 \end{align*} 
 Plugging these calculations into the high frequency bound for $\tilde{v}$ and taking the limit as $T\rightarrow\infty$ yields  $$\norm{u}_{\mathcal{LE}^1_\omega}+|\imag \omega|^{1/2}\norm{u}_{\dot{H}^1_\omega}\lesssim\norm{\inprod{x}^{-2}u}_{\mathcal{LE}}+\norm{f}_{\mathcal{LE}^*+|\imag\omega|^{1/2}L^2}.$$ 
 For sufficiently large $\omega$, the first term on the right absorbs into the first term on the left, which implies (\ref{le res bd}).
 \vskip .05in
 \noindent\textbf{Case II: $|\omega| \ll 1$}
\vskip .05in\noindent
 Since zero is not a resonance, Theorem \ref{zero res equiv} implies that (\ref{le res bd}) holds for all $\omega$ sufficiently close to 0. 
 \vskip .05in
 \noindent\textbf{Case III: $|\omega|\approx 1$} 
\vskip .05in\noindent
 Since there are no real resonances, Theorem \ref{res equiv} implies that (\ref{le res bd}) holds for all $\omega$ in this region. 
 \vskip .1in
 Hence, (\ref{le res bd}) holds.
\end{proof}
\bibliographystyle{amsplain}
\bibliography{main}

\begin{bibdiv}
\begin{biblist}

\bib{Bo18}{thesis}{
      author={Booth, R.},
       title={An investigation of non-trapping, asymptotically {E}uclidean wave equations},
        type={{PhD} dissertation},
        date={2018},
}

\bib{BR14}{article}{
      author={Bouclet, J.~M.},
      author={Royer, J.},
       title={Local energy decay for the damped wave equation},
        date={2014},
     journal={J. Funct. Anal.},
      volume={266},
      number={7},
       pages={4538\ndash 4615},
}

\bib{BT07}{article}{
      author={Bouclet, J.~M.},
      author={Tzvetkov, N.},
       title={Strichartz estimates for long range perturbations},
        date={2007},
     journal={Amer. J. Math},
      volume={129},
      number={6},
       pages={1565\ndash 1609},
}

\bib{BT08}{article}{
      author={Bouclet, J.~M.},
      author={Tzvetkov, N.},
       title={On global {S}trichartz estimates for non trapping metrics.},
        date={2008},
     journal={J. Funct. Anal.},
      volume={254},
      number={6},
       pages={1661\ndash 1682},
}

\bib{Hi22}{article}{
      author={Hintz, P.},
       title={A sharp version of {P}rice's law for wave decay on asymptotically flat spacetimes},
        date={2022},
     journal={Comm. Math. Phys.},
      volume={389},
      number={1},
       pages={491\ndash 542},
}

\bib{Hi23}{unpublished}{
      author={Hintz, P.},
       title={Linear waves on non-stationary asymptotically flat spacetimes. i.},
        date={2023},
        note={Preprint, arXiv:2302.14647},
}

\bib{JSS90}{article}{
      author={Journ\'e, J.~L.},
      author={Soffer, A.},
      author={Sogge, C.~G.},
       title={{${L}^p\to {L}^{p'}$} estimates for time dependent {S}chr\"{o}dinger operators},
        date={1990},
     journal={Bull. Amer. Math. Soc. (N.S.)},
      volume={23},
       pages={519\ndash 524},
}

\bib{JSS91}{article}{
      author={Journ\'e, J.~L.},
      author={Soffer, A.},
      author={Sogge, C.~G.},
       title={Decay estimates for {S}chr\"{o}dinger operators},
        date={1991},
     journal={Comm. Pure Appl. Math.},
      volume={44},
       pages={573\ndash 604},
}

\bib{Kof22}{article}{
      author={Kofroth, C.},
       title={Integrated local energy decay for the damped wave equation on stationary space-times},
        date={2023},
     journal={SIAM J. Math. Anal.},
      volume={55},
      number={5},
       pages={5086\ndash 5126},
}

\bib{KT01}{article}{
      author={Koch, H.},
      author={Tataru, D.},
       title={Carleman estimates and unique continuation for second-order elliptic equations with nonsmooth coefficients},
        date={2001},
     journal={Comm. Pure Appl. Math.},
      volume={54},
      number={3},
       pages={339\ndash 360},
}

\bib{Looi21}{article}{
      author={Looi, S.-Z.},
       title={Pointwise decay for the wave equation on nonstationary spacetimes},
        date={2023},
     journal={J. Math. Anal. Appl.},
       pages={126939},
}

\bib{MMT08}{article}{
      author={Marzuola, J.},
      author={Metcalfe, J.},
      author={Tataru, D.},
       title={Strichartz estimates and local smoothing estimates for asymptotically flat {S}chr\"{o}dinger equations},
        date={2008},
     journal={J. Funct. Anal},
      volume={255},
      number={6},
       pages={1497\ndash 1553},
}

\bib{MMTT10}{article}{
      author={Marzuola, J.},
      author={Metcalfe, J.},
      author={Tataru, D.},
      author={Tohaneanu, M.},
       title={Strichartz estimates on {S}chwarzschild black hole backgrounds},
        date={2010},
     journal={Comm. Math. Phys.},
      volume={293},
      number={1},
       pages={37\ndash 83},
}

\bib{Morg20}{article}{
      author={Morgan, Katrina},
       title={The effect of metric behavior at spatial infinity on pointwise wave decay in the asymptotically flat stationary setting},
        date={2024},
     journal={American Journal of Mathematics},
      volume={146},
      number={1},
       pages={47\ndash 105},
}

\bib{Mora66}{article}{
      author={Morawetz, C.~S.},
       title={Exponential decay of solutions of the wave equation},
        date={1966},
     journal={Comm. Pure Appl. Math.},
      volume={19},
       pages={439\ndash 444},
}

\bib{Mora68}{article}{
      author={Morawetz, C.~S.},
       title={Time decay for the nonlinear {K}lein-{G}ordon equations},
        date={1968},
     journal={Proc. Roy. Soc. London Ser. A},
      volume={306},
       pages={291\ndash 296},
}

\bib{Mora75}{article}{
      author={Morawetz, C.~S.},
       title={Decay for solutions of the exterior problem for the wave equation},
        date={1975},
     journal={Comm. Pure Appl. Math.},
      volume={28},
       pages={229\ndash 264},
}

\bib{MRS77}{article}{
      author={Morawetz, C.~S.},
      author={Ralston, J.~V.},
      author={Strauss, W.~A.},
       title={Decay of solutions of the wave equation outside nontrapping obstacles},
        date={1977},
     journal={Comm. Pure Appl. Math.},
      volume={30},
      number={4},
       pages={447\ndash 508},
}

\bib{MST20}{article}{
      author={Metcalfe, J.},
      author={Sterbenz, J.},
      author={Tataru, D.},
       title={Local energy decay for scalar fields on time dependent non-trapping backgrounds},
        date={2020},
     journal={Amer. J. Math.},
      volume={142},
      number={3},
       pages={821\ndash 883},
}

\bib{MT09}{incollection}{
      author={Metcalfe, J.},
      author={Tataru, D.},
       title={Decay estimates for variable coefficient wave equations in exterior domains},
        date={2009},
   booktitle={Advances in phase space analysis of partial differential equations},
      series={Progr. Nonlinear Differential Equations Appl.},
      volume={78},
       pages={201\ndash 216},
}

\bib{MT12}{article}{
      author={Metcalfe, J.},
      author={Tataru, D.},
       title={Global parametrices and dispersive estimates for variable coefficient wave equations},
        date={2012},
     journal={Math. Ann.},
      volume={353},
      number={4},
       pages={1183\ndash 1237},
}

\bib{MTT12}{article}{
      author={Metcalfe, J.},
      author={Tataru, D.},
      author={Tohaneanu, M.},
       title={Price's law on nonstationary space-times},
        date={2012},
     journal={Adv. Math.},
      volume={230},
      number={3},
       pages={995\ndash 1028},
}

\bib{MW21}{unpublished}{
      author={Morgan, K.},
      author={Wunsch, J.},
       title={Generalized {P}rice's law on fractional-order asymptotically flat stationary spacetimes},
        date={2021},
        note={Accepted for publication in Math. Res. Lett., arXiv:2105.02305},
}

\bib{RT74}{article}{
      author={Rauch, J.},
      author={Taylor, M.~E.},
       title={Exponential decay of solutions to hyperbolic equations in bounded domains},
        date={1974},
     journal={Indiana Univ. Math. J.},
      volume={24},
      number={1},
       pages={79\ndash 86},
}

\bib{Tat08}{article}{
      author={Tataru, D.},
       title={Parametrices and dispersive estimates for {S}chr\"{o}dinger operators with variable coefficients},
        date={2008},
     journal={Amer. J. Math.},
      volume={130},
      number={3},
       pages={571\ndash 634},
}

\bib{Tat13}{article}{
      author={Tataru, D.},
       title={Local decay of waves on asymptotically flat stationary space-times},
        date={2013},
     journal={Amer. J. Math.},
      volume={135},
      number={2},
       pages={361\ndash 401},
}

\bib{Toh12}{article}{
      author={Tohaneanu, M.},
       title={Strichartz estimates on {K}err black hole backgrounds},
        date={2012},
     journal={Trans. Amer. Math. Soc.},
      volume={364},
      number={2},
       pages={689\ndash 702},
}

\end{biblist}
\end{bibdiv}
\end{document}